\documentclass{amsart}
\usepackage{amsfonts,amssymb,latexsym}
\usepackage{amsmath,amsthm}
\usepackage{eucal}
\usepackage[latin1]{inputenc}
\usepackage{color}

\newtheorem{theorem}{Theorem}[section]
\newtheorem{lemma}[theorem]{Lemma}
\newtheorem{proposition}[theorem]{Proposition}

\theoremstyle{definition}
\newtheorem{definition}[theorem]{Definition}

\theoremstyle{remark}
\newtheorem{remark}{Remark}[section]

\numberwithin{equation}{section}

\begin{document}

\title[Local well-posedness for the KdV hierarchy at high regularity]{Local well-posedness for the KdV hierarchy at high regularity}
\author[C. E. Kenig and D. Pilod]{Carlos E. Kenig $^{\star}$ and Didier Pilod $^{\dagger}$}
\thanks{$^{\star}$ Partially supported by NSF Grants DMS-1463746 and DMS-1265249.}
\thanks{$^{\dagger}$ Partially supported by CNPq/Brazil, Grant 302632/2013-1 and 481715/2012-6.}
\subjclass[2010]{Primary 35Q53,  37K10,  35A01; Secondary 37K05, 35E15 }
\keywords{KdV hierarchy, local well-posedness, modified energy}

\keywords{}

\maketitle

\vspace{-0.5cm}

{\scriptsize \centerline{$^{\star}$Department of Mathematics, University of Chicago, }
             \centerline{ Chicago, IL, 60637 USA.}
             \centerline{email: cek@math.chicago.edu}

\vspace{0.1cm}
             \centerline{$^{\dagger}$  Instituto de Matem\'atica,
             Universidade Federal do Rio de Janeiro,}
              \centerline{Caixa Postal 68530, CEP: 21945-970, Rio
              de Janeiro, RJ, Brazil.}
              \centerline{email: didierpilod@gmail.com}}

\vspace{0.5cm}

\begin{abstract} 
We prove well-posedness in $L^2$-based Sobolev spaces $H^s$ at high regularity for a class of nonlinear higher-order dispersive equations generalizing the KdV hierarchy  both on the line and on the torus.
\end{abstract}

\section{Introduction}

The Korteweg-de Vries (KdV) equation 
\begin{equation} \label{KdV}
\partial_tu+\partial_x^3u=u\partial_xu \, 
\end{equation}
is well-known to be a completely integrable system. In particular Kruskal, Gardner and Miura \cite{GaKrMi} have constructed explicitly an infinite sequence of functionals $H_l(u)$, $l \in \mathbb N$, which are constant along the flow of \eqref{KdV}. Moreover, Gardner \cite{Ga} has shown that all the  functionals $H_l(u)$ are also constant along the flow of the generalized class of equations 
\begin{equation}  \label{KdVhierarchy}
\partial_tu=\partial_x G_l(u), \quad l \in \mathbb N ,
\end{equation}
introduced by Lax \cite{La} and called the \emph{KdV hierarchy}. Here $G_l(u)$ is defined by the induction formula 
\begin{displaymath}
\left\{\begin{array}{l} \partial_xG_{l+1}(u)=\left(\partial_x^3+\frac23u\partial_x+\frac13u_x \right) G_l(u), \quad l \ge 1, \\ G_0(u)=u. \end{array}\right.
\end{displaymath}
In particular, each equation in \eqref{KdVhierarchy} has a Hamiltonian structure associated to the Hamiltonian $H_l(u)$, defined by $\text{grad} \, H_l(u)= G_l(u)$. Observe that the equation in \eqref{KdVhierarchy} corresponding to $l=0$ is the linear wave equation, while the one corresponding to $l=1$ is the KdV equation. We will call the equation corresponding to $l=2$, 
\begin{equation} \label{5KdV}
\partial_tu-\partial_x^5u=\partial_x\Big(\frac53u\partial_x^2u+\frac56(\partial_xu)^2+\frac5{18}u^3 \Big)\, ,
\end{equation}
the fifth-order KdV equation, since it has $5$ derivatives in the linear part. We also refer to the nice introductions in \cite{Gr,Po,Sa} for more details and references on the KdV hierarchy.

In \cite{Sa}, Saut used the Hamiltonian structure to prove the existence of global weak solutions to \eqref{KdVhierarchy} in the energy space $H^l(\mathbb R)$, for each $l \in \mathbb N$, $l \ge 1$. Later on, Schwarz \cite{Sch} studied the class of equations \eqref{KdVhierarchy} in the periodic setting (\textit{i.e.} when the space variable $x \in \mathbb T$). He showed uniqueness of the solutions to \eqref{KdVhierarchy} associated to initial data in $H^n(\mathbb T)$, for $n \ge 3l+1$. His proof relies on the use of a modified energy, related to the Hamiltonian $H_l$, to control the difference of two solutions at the $H^l$-level. The proof seems to apply also very well in the continuous setting (\textit{i.e.} when $x \in \mathbb R$).

\medskip
Our purpose in this article  is to study local well-posedness of the initial value problem (IVP) associated to the whole KdV hierarchy \eqref{KdVhierarchy} in $L^2$-based Sobolev spaces $H^s$ at high regularity (for $s$ large enough) in both the continuous and periodic cases. The notion of well-posedness to be used here includes  
existence, uniqueness, persistence property (\textit{i.e.} the solution $u$ describes a continuous curve in $H^s$ whenever the initial datum $u_0=u(\cdot,0)$ belongs to $H^s$) and continuous dependence of the flow upon the initial data. In other words, we shall say that the IVP associated to \eqref{KdVhierarchy} is well-posed in $H^s$ if it induces a dynamical system on $H^s$ by generating a continuous local flow.

Actually, we will prove our result for the more general class of IVPs associated to the higher-order nonlinear dispersive equations 
\begin{equation} \label{hodisp}
\partial_tu+c_{2l+1}\partial_x^{2l+1}u+\sum_{k=2}^{l+1}N_{lk}(u)=0 \, ,
\end{equation}
where $x \in \mathbb R$ or $\mathbb T$, $t \in \mathbb R$, $u=u(x,t) \in \mathbb R$, $l \in \mathbb N,$ $l \ge 1$, $c_{2l+1} \neq 0$ and 
\begin{equation} \label{nonlinearity}
N_{lk}(u)=\sum_{|n|=2(l-k)+3}c_{l,k,n}\partial_x^{n_0}\prod_{i=1}^k\partial_x^{n_i}u\, ,
\end{equation}
with $|n|=\sum_{i=0}^kn_i$, $n_i \in \mathbb N,$ for $i=0,\cdots, k$ and $c_{l,k,n} \in \mathbb R$. 

This class of equations is similar\footnote{Actually, the only difference with the class introduced in \cite{Gr} is that we do not need to assume that the nonlinearity is in divergence form, \textit{i.e} here we assume $n_0 \ge 0$ instead of $n_0 \ge 1$ as in \cite{Gr}.} to the one introduced by Gr\"unrock in \cite{Gr} and generalizes the KdV hierarchy.  However, the equations in the class are not necessarily completely integrable or even hamiltonian. 

If we define the rank $r$ of a monomial $\partial_x^{n_0}\prod_{i=1}^k\partial_x^{n_i}u$ by $r=k+\frac{|n|}2$ where $k$ is the number of factors and $|n|=\sum_{i=0}^kn_i$ is the total number of differentiations, we observe that all the monomials appearing in  the nonlinearities \eqref{nonlinearity} of \eqref{hodisp} have the same rank $r=l+\frac32$.  For the quadratic terms corresponding to $k=2$, the total number of differentiations is then $2l-1$ and we need to deal with terms of the form $u\partial_x^{2l-1}u$. They are the most difficult terms to handle since they display a high-low frequency interaction in the nonlinearity as we will see below.

Moreover, the equations in \eqref{hodisp} are invariant under the scaling transformation $u_{\lambda}(x,t)=\lambda^2u(\lambda x,\lambda^{2l+1}t)$ for any $\lambda >0$ with initial data $u_{\lambda}(x,0)=\lambda^{2}u(\lambda x,0)$. Hence, $\|u_{\lambda}(\cdot,0)\|_{\dot{H}^s}=\lambda^{\frac32+s}\|u(\cdot,0)\|_{\dot{H}^s}$. Consequently, the critical Sobolev index for the class of equations \eqref{hodisp}  is $s_c = -\frac32$ just as for the KdV equation.

\medskip

In the case $l=1$, the IVP associated to the KdV equation has already been extensively studied. In particular it has been shown to be well-posed in the energy space  $H^1(\mathbb R)$ by Kenig, Ponce and Vega \cite{KPV3} in the continuous case (see also \cite{Bo2,KPV4,CKSTT,Guo,Kish} for further results at lower regularity) and in $L^2(\mathbb T)$ by Bourgain \cite{Bo2}  in the periodic case (see also \cite{KPV4,CKSTT} for further results at lower regularity). Moreover, since this result can be proved by using a fixed point argument in well-suited function spaces (related to the dispersive properties of the associated linear equation), the flow map of \eqref{KdV} is smooth. In other words, the KdV equation exhibits a \emph{semi-linear nature}.

This last property is however not true anymore for the other equations in the hierarchy (corresponding to $l \ge 2$) on the line. Indeed, it was proved by the second author \cite{Pi}, by adapting an argument of Molinet, Saut and Tzvetkov for the Benjamin-Ono equation \cite{MST1}, that the flow map associated to \eqref{hodisp}-\eqref{nonlinearity} fails to be $C^2$ in $L^2$-based Sobolev space $H^s(\mathbb R)$ for any $s 
\in \mathbb R$. This is due to the lack of control of the high-low frequency interaction in nonlinear quadratic terms of the form $\partial_x(u\partial_x^{2l-2}u)$ or $u\partial_x^{2l-1}u$. Note that strictly speaking
the proof in \cite{Pi} was given only in special cases of equations having only quadratic nonlinearities,  but, as was observed Gr\"unrock \cite{Gr}, since the cubic and higher terms in \eqref{nonlinearity} are well behaved, no cancellations occur, and the proof applies as well to the more general class of nonlinearities \eqref{nonlinearity} provided that the coefficient in front of 
$\partial_x(u\partial_x^{2l-2}u)$ or $u\partial_x^{2l-1}u$ is not $0$. In this sense, the equations following KdV in the KdV hierarchy \eqref{KdVhierarchy} exhibit a \emph{quasi-linear nature}. As a consequence, one cannot  solve the IVPs associated to \eqref{hodisp} by a Picard iterative method implemented
on the associated integral equations  for initial data in any Sobolev space $H^s(\mathbb R)$ with $s \in \mathbb R$.

However, the fixed point method may still be employed to prove well-posedness for \eqref{5KdV} in other function spaces. For example in \cite{KPV3,KPV4}, Kenig, Ponce and Vega dealt with the more general class of IVPs
\begin{equation} \label{KdV_h}
\left\{  \begin{array}[pos]{ll}
          \partial_tu+\partial_x^{2l+1}u=P(u,\partial_xu,\ldots,\partial_x^{2l}u),
          \quad x,\ t \in \mathbb R, \ l \in \mathbb N, l \ge 1 \\          u(0)=u_0, \\
     \end{array} \right.
\end{equation}
where
\begin{displaymath}
P: \mathbb R^{2l+1}\rightarrow \mathbb R \quad (\mbox{or} \ P:
\mathbb C^{2l+1}\rightarrow \mathbb C)
\end{displaymath}
is a polynomial having no constant or linear terms. They proved well-posedness in weighted Sobolev spaces of the type $H^k(\mathbb R) \cap H^m(\mathbb R; x^2dx)$ with $k, \ m \in \mathbb Z_+$, $k \ge k_0, \ m \ge m_0$ for some $k_0, \ m_0 \in \mathbb Z_+$ large enough. We also refer to \cite{Pi} for sharper results in the case of small initial data and when the nonlinearity in \eqref{KdV_h} is quadratic and to \cite{KeSt} for local well-posedness for a class of systems on the form \eqref{KdV_h}.

Recently, Gr\"unrock \cite{Gr} used a variant of the Fourier restriction norm method to prove well-posedness for the whole class of equations \eqref{hodisp}\footnote{There is also a sharp well-posedness result in the case $l=1$.} with $l \ge 2$ in $\widehat{H}^s_r(\mathbb R)$ for $1<r \le \frac{2l}{2l-1}$ and $s>l-\frac32-\frac1{2l}+\frac{2l-1}{2r'}$. Here, the space $\widehat{H}^s_r(\mathbb R)$ is defined by the norm
$ \|\varphi\|_{\widehat{H}^s_r}=\|\langle \xi\rangle^s \widehat{\varphi}\|_{L^{r'}} $with $\frac1r+\frac1{r'}=1$. We also refer to the work of Kato \cite{Ka} for a well-posedness result using another variant of the Fourier norm restriction method in the specific case $l=2$. He showed that the corresponding IVP is well-posed in $H^{s,a}(\mathbb R)$ for $s \ge \max\{-\frac14,-2a-2\}$ with $-\frac32 < a \le -\frac14$ and $(s,a) \neq (-\frac14,-\frac78)$, $H^{s,a}(\mathbb R)$ is a $H^s$-type space with a weight on low frequency and is defined by the norm
 $\|\varphi\|_{H^{s,a}}=\|\langle \xi\rangle^{s-a}|\xi|^a \widehat{\varphi}\|_{L^2}$.

\medskip
Nevertheless, the $L^2$-based Sobolev spaces $H^s$ remain the natural spaces to
study well-posedness for the the class of higher-order nonlinear dispersive equations \eqref{hodisp}, since when those equations possess a Hamiltonian structure, it is well-defined for functions in $H^l$ (as for example the equations in the KdV hierarchy \eqref{KdVhierarchy}). 

Now, we recall the results concerning the well-posedness in $H^s(\mathbb R)$ of \eqref{hodisp} in the case $l=2$. Due to the ill-posedness result in \cite{Pi}, we need to use an alternative method to the Picard iteration. The direct energy estimate for equation \eqref{hodisp} with $l=2$ (considering only the quadratic terms for simplicity) gives only
\begin{equation} \label{standardenergy}
\frac{d}{dt}\|\partial_x^ku(t)\|_{L^2}^2 \lesssim \|\partial_x^3u\|_{L^{\infty}_x}\|\partial_x^ku(t)\|_{L^2} ^2+\Big|\int_{\mathbb R}\partial_xu\partial_x^{k+1}u\partial_x^{k+1}udx\Big|.
\end{equation}
Observe that the last term on the right-hand side of \eqref{standardenergy} has still higher-order derivatives and cannot be treated by using only integration by parts. To overcome this difficulty, Ponce \cite{Po} used a recursive argument based on the dispersive smoothing effects associated to the linear part of \eqref{5KdV}, combined with a parabolic regularization method, to establish that the IVP \eqref{hodisp} is locally well-posed in $H^s(\mathbb R)$ for $s \ge 4$ in the case $l=2$. Later, Kwon \cite{Kw} improved Ponce's result by proving local well-posedness in $H^s(\mathbb R)$ for $s>5/2$. The main new idea was to modify the energy by adding a lower-order cubic term correction to cancel the last term on the right-hand side of \eqref{standardenergy}. Note that he also used the dispersive smoothing properties of the linear part in order to refine the argument.
Finally, the authors \cite{KP}, and independently Guo, Kwak and Kwon \cite{GKK}, proved recently that the IVP associated to \eqref{hodisp} with $l=2$ is well-posed in the energy space $H^2(\mathbb R)$. In \cite{KP}, we followed the method introduced by Ionescu, Kenig and Tataru \cite{IKT} in the context of the KP1 equation, which is based on the \lq\lq short time\rq\rq \, dyadic $X^{s,b}$ spaces. Moreover, in order to derive the crucial energy estimate, we used a modified energy defined at the dyadic level. Guo, Kwak and Kwon also used the \lq\lq short-time\rq\rq \, $X^{s,b}$ method. However, instead of modifying the energy as we did, they put an additional weight in the $X^{s,b}$ structure of the spaces in order to derive the key energy estimate. 

In the case of the fifth-order KdV equation, we would like also to mention the works \cite{Daw} for unique continuation properties (see also \cite{Isa} in the case of the KdV hierarchy), \cite{IsLiPo} for decay properties and \cite{SeSm} for the propagation of regularity.

\medskip
There are, as far as we know, no \lq\lq complete\rq\rq \,  well-posedness results for the class of equations \eqref{hodisp} in $L^2$-based Sobolev spaces $H^s$  on the line when $l\ge 3$ and on the torus when $l \ge 2$. The aim of this paper is to fill (partially) this gap by proving the following local well-posedness result for all $l \in \mathbb N$, $l \ge 2$ at high regularity on both the line and the torus.

\begin{theorem} \label{localtheory}
We denote by $M=\mathbb R$ or $\mathbb T$. Assume that $l \ge 2$ and let $s>s_l=4l-\frac92$. Then for all  $u_0 \in H^s(M)$, there exists a positive time $T=T(\|u_0\|_{H^s})$ and a unique solution $u$ to \eqref{hodisp} in the class 
$C([0,T]:H^s(M))$ satisfying $u(\cdot,0)=u_0$. Moreover, for any $0<T'<T$, there exists a neighborhood $\mathcal{U}$ of $u_0 \in H^s(M)$ such that the flow map data-solution 
\begin{equation} \label{localtheory.1}
S_{T'}^s: \mathcal{U} \longrightarrow C([0,T']:H^s(M)), \ u_0 \mapsto u \, 
\end{equation}
is continuous.
\end{theorem}

\begin{remark} Of course, in the case $l=2$ and $M=\mathbb R$, local well-posedness results are already known in $H^s(\mathbb R)$ for $s \ge 4$ in \cite{Po}, $s>\frac52$ in \cite{Kw} and $s \ge 2$ in \cite{KP,GKK}.
\end{remark}

\begin{remark} 
Although Schwarz's uniqueness result \cite{Sch} for the KdV hierarchy \eqref{KdVhierarchy} is obtained at lower regularities for large $l$ than in Theorem \ref{localtheory}, his result does not seem to apply when perturbating the coefficients in \eqref{KdVhierarchy}, since it depends on the Hamiltonian structure of the equations. On the other hand, the uniqueness result in Theorem \ref{localtheory} holds for the whole class \eqref{hodisp}. 

Moreover, we also prove the continuous dependence of the flow, providing the first \lq\lq complete\rq\rq \,  well-posedness result for the KdV hierarchy in $L^2$ based Sobolev spaces $H^s$.
\end{remark}

When proving Theorem \ref{localtheory}, we adapt  Kwon's modified energy argument  \cite{Kw} for the fifth-order KdV equation \eqref{5KdV} to the higher-order equations in \eqref{hodisp}. For the sake of simplicity, we will work with the equation 
\begin{equation} \label{homodel}
\partial_tu+\partial_x^{2l+1}u=u\partial_x^{2l-1}u \, , \quad l \ge 2 \, ,
\end{equation}
which is a particular case of \eqref{hodisp} and whose nonlinearity $u\partial_x^{2l-1}u$ presents the worst type of high-low frequency interactions in \eqref{nonlinearity}. 

Our proof is based on energy estimates. By using higher-order commutator estimates, we obtain that
\begin{equation} \label{homodel.1}
\frac12\frac{d}{dt}\|u\|_{H^s}^2=\mathcal{O}(\|\partial_x^{l-1}u\|_{L^{\infty}}\|u\|_{H^s}^2)+\sum_{j=1}^{l-1}\beta_j\int \partial_x^{2(l-j)-1}u(D^s\partial_x^ju)^2 \, ,
\end{equation}
where $\beta_j$, $j=1\cdots l-1,$ are real numbers. Since the $l-1$ terms appearing in  the right-hand side of \eqref{homodel.1} cannot be handled directly by integration by parts, we need to add $l-1$ correcting cubic terms to the energy in order to cancel them out. We then modify the energy as $E^s(u)=\frac12\|u\|_{H^s}^2+T^s_3(u),$ where $T^s_3(u)=\sum_{j=0}^{l-2}\gamma_jT^s_{3,j}(u)$ and
\begin{displaymath} 
T_{3,j}^s(u)=\left\{\begin{array}{ll}
\int\partial_{x}^{2j}u(D^{s-2-j}\partial_xu)^2 & \text{if} \ j \ \text{is even} \\ 
\int\partial_{x}^{2j}u(D^{s-1-j}u)^2 & \text{if} \ j \ \text{is odd} 
\end{array} \right.  , \ \text{for} \ 0 \le j \le l-2 .
\end{displaymath} 
Let us denote by $X^s_{3 \to 3}(u)$, respectively $X^s_{3 \to 4}(u)$, the cubic terms coming from the linear part of \eqref{homodel}, respectively the fourth-order terms coming from the nonlinear part of \eqref{homodel}, when deriving $T_3^s(u)$. In other words, we have that
$$\frac{d}{dt}T^s_3(u)=X^s_{3 \to 3}(u)+X^s_{3 \to 4}(u) \, .$$
By choosing carefully the coefficients $\gamma_{j}$, $0 \le j \le l-2$, we obtain 
\begin{displaymath} \left| \frac12\frac{d}{dt}\|u\|_{H^s}^2+X^s_{3 \to 3}(u) \right| \lesssim \Big(\sum_{j=0}^{l-2}\|\partial_x^{2l-1+2j}u\|_{L^{\infty}}\Big)\|u\|_{H^s}^2 \, .
\end{displaymath}
Then, we use the bound $\sum_{j=0}^{l-2}\|\partial_x^{2l-1+2j}u\|_{L^{\infty}} \lesssim \|u\|_{H^s}$, due to the Sobolev embedding. This provides the condition  $s>4l-\frac92$ in Theorem \ref{localtheory}. It remains to control $X^s_{3 \to 4}(u)$. In the case $l=2$, the bound follows directly from the Kato-Ponce commutator estimate. In the case $l \ge 3$, we repeat the argument and  modify the energy by a fourth order term $T^s_4(u)$ in order to cancel the bad terms in $X^s_{3 \to 4}(u)$. This process will finish after $l-1$ steps and we will obtain an energy estimate
\begin{displaymath}
\Big|\frac{d}{dt}E^s(u)(t)\Big|  \lesssim \left(\sum_{k=1}^{l-1}\|u(t)\|_{H^s}^k\right)\|u(t)\|_{H^s}^2 \, ,
\end{displaymath}
for an energy of the form $E^s(u)=\frac12\|u\|_{H^s}^2+T^s_3(u)+\cdots+T^s_{l+1}(u)$, where $T_j^s(u)$ are terms of order $j$ in $u$ for $j=3,\cdots,l+1$.

This energy is coercive if $\|u\|_{H^s}$ is small enough. Moreover, by using a scaling argument, it is always possible to assume that the initial data are sufficiently small. Hence, we deduce \textit{a priori} estimates for solutions of \eqref{homodel} at the $H^s$-level . The proof of the existence follows then by a classical parabolic regularization argument.

The same modified energy argument also applies to derive energy estimates for the differences of two solutions of \eqref{homodel} at the $L^2$-level and also at the $H^s$-level (see Proposition \ref{diffEE} below). Then, the  proof of the uniqueness follows directly from the $L^2$-energy estimate, while we need to combine the $H^s$-energy estimate with the Bona-Smith argument \cite{BoSm} in order to deduce the persistence property and the continuity of the flow.

Note that in the continuous setting (when $x \in \mathbb R$), it would be possible to use the dispersive properties of the linear part of \eqref{homodel} as in \cite{Kw} in order to lower the regularity $s>4l-\frac92$ in Theorem \ref{localtheory}. We did not pursue this issue here. However, it is worth to note that, since our proof uses only energy estimates and Sobolev embedding, it applies similarly in the periodic setting.

\medskip
Finally, we would like to point out that our argument does not seem to apply to the whole class of equations \eqref{KdV_h} considered in \cite{KPV2}. For example, if we want to deal with the equation
 \begin{equation} \label{hoKdV}
 \partial_tu+\partial_x^3u+u\partial_x^2u=0 \, ,
 \end{equation}
we would need to modify the energy by adding a term of the form $\int \partial_x^{-1}u(D^{s}u)^2$, which would not be well defined at the $H^s$-level\footnote{Indeed, their is a problem  to define $\partial_x^{-1}u$ at low frequency.}. Note that the $C^2$ ill-posedness result in \cite{Pi} also apply to the IVP associated to \eqref{hoKdV} in any Sobolev spaces $H^s$. Recently Harrop-Griffiths \cite{HaGr} obtained well-posedness results for \eqref{hoKdV} with initial data in a translation invariant space $l^1H^s \subset H^s$. However, when the nonlinearity has a special structure, as for example in the following higher-order Benjamin-Ono equation 
\begin{equation} \label{hoBO}
\partial_tv-\mathcal{H}\partial^2_xv-
           \epsilon \partial_x^3v=v\partial_xv-\epsilon
           \partial_x(v\mathcal{H}\partial_xv+\mathcal{H}(v\partial_xv)) \, ,
\end{equation}
it is still possible to obtain well-posedness results in $H^s$ spaces (see  \cite{LiPiPo} and \cite{MoPi} where global well-posedness for \eqref{hoBO} was obtained in the energy space $H^1(\mathbb R)$).

\bigskip
The paper is organized as follows: in Section \ref{PrelimEst}, we derive a key technical Lemma based on integration by parts and state the Kato-Ponce type commutator estimates both on the line and on the torus. Those results will be useful to derive the estimates in Section \ref{EnergyEst}. With the energy estimates in hand, we give the proof of Theorem \ref{localtheory} in Section \ref{ProofTheo} by using the parabolic regularization method. Note that for the sake of clarity, we chose to derive the energy estimates for the equation without dissipation in Section \ref{EnergyEst}. We explain then how to adapt the proof in presence of dissipation in Section \ref{ProofTheo}.

\bigskip
\noindent \textbf{Notation.} 
For any positive numbers $a$ and $b$, the notation $a \lesssim b$ means
that there exists a positive constant $C$ such that $a \le C b$,
and we denote $a \sim b$ when, $a \lesssim b$ and $b \lesssim a$.
We will also denote by $C$ any universal positive constant. In particular, the value of $C$ can change from line to line.

Let $\mathcal{S}(\mathbb R)$ be the Schwartz space and $\mathcal{P}$ the space of $C^{\infty}$, $2\pi$ periodic functions on $\mathbb R$. Then, the space of tempered distributions $\mathcal{S}'(\mathbb R)$ is the dual space of $\mathcal{S}(\mathbb R)$, and the space of periodic distributions $\mathcal{P}'$ is the dual space of $\mathcal{P}$.

For $f \in \mathcal{S}'(\mathbb R)$ or $\mathcal{P}'$, $\mathcal{F}f=\widehat{f}$ will denote its  Fourier
transform. 

Let $s \in \mathbb R$,  $f \in \mathcal{S}(\mathbb R)$ or $\mathcal{P}$, we define the Bessel and Riesz potentials of order $-s$, $J^s$ and $D^s$, by
\begin{displaymath}
J^sf=\mathcal{F}^{-1}\big((1+|\xi|^2)^{\frac{s}{2}}
\mathcal{F}(f)\big) \quad \text{and} \quad
D^sf=\mathcal{F}^{-1}\big(|\xi|^s \mathcal{F}(f)\big).
\end{displaymath}

 For $M=\mathbb R$ or $\mathbb T=\mathbb R/2\pi \mathbb Z$  and $1 \le p\le +\infty$, $L^p(M)$ denotes the usual Lebesgue space on $M$ with associated norm $\|\cdot\|_{L^p}$ or $\|\cdot\|_{L^p(M)}$ when we want to differentiate the cases $M=\mathbb R$ or $\mathbb T$.

For $M=\mathbb R$ or $\mathbb T$, $H^s(M)$ denotes  the nonhomogeneous Sobolev space defined as the completion of $\mathcal{S}'(\mathbb R)$ or $\mathcal{P}'$ under the norm
$\|f\|_{H^s}=\|J^sf\|_{L^2}$.

If $B$ is one of the spaces defined above, $1 \le p \le +\infty$ and $T>0$, we define the spaces $L^p_tB_x=L^p([0,+\infty) : B)$ and $L^p_TB_x=L^p([0,T] : B)$ equipped with the norms
$$\|u\|_{L^p_tB_x}=\Big(\int_0^{+\infty}\|u(\cdot,t)\|_{B_x}^pdt\Big)^{\frac1p} \quad \text{and} \quad 
\|u\|_{L^p_TB_x}=\Big(\int_0^{T}\|u(\cdot,t)\|_{B_x}^pdt\Big)^{\frac1p}$$
with obvious modifications for $p=+\infty$.

We introduce the operators $P_{low}$ and $P_{high}$ of projection into low and high frequencies. Let $\eta \in C^{\infty}_c(\mathbb R)$ be an even cut-off function satisfying $0 \le \eta \le 1$, $\text{supp}\, \eta \subset [-2,2]$ and $\eta_{_{[-1,1]}} \equiv 1$. Then we define $P_{low}$ and $P_{high}$ on $\mathbb R$ \textit{via} Fourier transform by 
\begin{equation} \label{Plow}
P_{low}f=\big(\eta \widehat{f}\big)^{\vee} \quad \text{and} \quad P_{high}f=(1-P_{low})f \, .
\end{equation}
In the periodic case, $P_{low}$ is simply defined by $(P_{low}f)^{\wedge}=\widehat{f}(0)$.
We also denote $f_{low}=P_{low}f$ and $f_{high}=P_{high}f$. It turns out that 
$$ \|f\|_{H^s} \sim \|f_{low}\|_{L^2}+\|D^sf_{high}\|_{L^2} \, ,$$
for any function $f \in H^s(M)$.

\section{Preliminary estimates} \label{PrelimEst}

\subsection{Technical lemma}[Integration by parts]

\begin{definition} \label{I}
Let $f, \ g,$ and $h$ be smooth functions defined on $\mathbb R$ or $\mathbb T$. For any $l \in \mathbb N$, we define 
\begin{equation} \label{I.1}
I_{2l+1}\big(w,f,g\big)=\int \partial_x^{2l+1}w f g+\int w \partial_x^{2l+1}f g+\int w  f\partial_x^{2l+1}g \ .
\end{equation}
\end{definition}

\begin{lemma} \label{technical}
We have that $I_1(w,f,g)=0$.  For $l \in \mathbb N$, $l \ge 1$, there exist real numbers $\alpha_{j,l}$ for $1 \le j \le l$ such that
\begin{equation} \label{technical.1}
I_{2l+1}(w,f,g)=\sum_{j=1}^l \alpha_{j,l} \int \partial_x^{2(l-j)+1}w \partial_x^jf \partial_x^j g \ .
\end{equation}
Moreover 
\begin{equation}\label{technical.2}
\alpha_{l,l}=(-1)^{l+1}(2l+1) \quad \text{ for} \ l \ge 1.
\end{equation}
\end{lemma}

\begin{proof} 
First, observe integrating by parts that 
\begin{displaymath}
I_1(w,f,g)=-\int  w \partial_x(f g)+\int w \partial_x f g + \int w f \partial_x g =0 \ .
\end{displaymath}

Next, fix some $l \in \mathbb N$, $l \ge 1$ and assume that \eqref{technical.1} holds true for $1 \le k \le l-1$. Then, we  integrate by parts and use the Leibniz rule to obtain
\begin{equation} \label{technical.3}
\begin{split}
I_{2l+1}(w,f,g) & =-\int w \partial_x^{2l+1}(f g)+\int w \partial_x^{2l+1}f g+\int w  f\partial_x^{2l+1}g \\ 
& =- \sum_{j=1}^{2l}\begin{pmatrix} 2l+1 \\ j \end{pmatrix} \int w \partial_x^jf \partial_x^{2l+1-j}g \ ,
\end{split}
\end{equation}
where $\begin{pmatrix} n \\ k \end{pmatrix}=\frac{n!}{k!(n-k)!}$. We rearrange the terms on the right hand-side of \eqref{technical.3} two by two by using that $\begin{pmatrix} n \\ k \end{pmatrix}=\begin{pmatrix} n \\ n-k \end{pmatrix}$ so that
\begin{equation} 
\begin{split}
I_{2l+1}(w,f,g)&=-\sum_{j=1}^l \begin{pmatrix} 2l+1 \\ j \end{pmatrix} \left( \int w \partial_x^jf \partial_x^{2l+1-j}g+\int w \partial_x^{2l+1-j}f \partial_x^{j}g\right) \\ &
=-\sum_{j=1}^l \begin{pmatrix} 2l+1 \\ j \end{pmatrix}\left(I_{2l-(2j-1)}(w,\partial_x^jf,\partial_x^jg)-\int \partial_x^{2l-(2j-1)}w\partial_x^jf\partial_x^jg \right),
\end{split}
\end{equation}
which proves \eqref{technical.1} in the case $l$ by using \eqref{technical.1} in the case $1 \le k \le l-1$. This concludes the proof of formula \eqref{technical.1} by induction. 

To simplify the notations, let us denote $a_l=\alpha_{l,l}$. Clearly, $a_1=3$. We also observe from the above construction that 
\begin{equation} \label{technical.4}
a_l=\begin{pmatrix} 2l+1 \\ l \end{pmatrix}-\sum_{j=1}^{l-1}\begin{pmatrix} 2l+1 \\ j \end{pmatrix}a_{l-j}, \quad \text{for} \ l \ge 2 \, .
\end{equation}
Next, we prove formula \eqref{technical.2} by induction. Let $l \ge 2$. Assume that formula \eqref{technical.2} is true for all $1 \le j \le l-1$. Without loss of generality, we assume that $l$ is even, $l=2l'$. Then, we deduce from \eqref{technical.4}  that
\begin{displaymath} \label{technical.5}
\begin{split}
a_l&=-\sum_{j=1}^l \begin{pmatrix} 2l+1 \\ j \end{pmatrix} (-1)^{j+1}(2l-2j+1) \\ 
& =\sum_{j'=1}^{l'}\begin{pmatrix} 2l+1 \\ 2j' \end{pmatrix} (2l-4j'+1)-
\sum_{j'=1}^{l'}\begin{pmatrix} 2l+1 \\ 2j' -1\end{pmatrix} (2l-4j'+3) \\ 
& =(2l+1)\sum_{j'=1}^{l'}\left\{\begin{pmatrix} 2l+1 \\ 2j' \end{pmatrix}-\begin{pmatrix} 2l+1 \\ 2j'-1 \end{pmatrix}\right\}\\ & \quad -2\sum_{j'=1}^{l'}\begin{pmatrix} 2l+1 \\ 2j' \end{pmatrix}2j'+2\sum_{j'=1}^{l'}\begin{pmatrix} 2l+1 \\ 2j'-1 \end{pmatrix}(2j'-1) \\ 
& =(2l+1)\sum_{j'=1}^{l'}\left\{\begin{pmatrix} 2l+1 \\ 2j' \end{pmatrix}-\begin{pmatrix} 2l+1 \\ 2j'-1 \end{pmatrix}-2\begin{pmatrix} 2l \\ 2j'-1 \end{pmatrix}+2\begin{pmatrix} 2l\\ 2j'-2 \end{pmatrix}\right\} \, .
\end{split}
\end{displaymath}
Since $\begin{pmatrix} n \\ k \end{pmatrix}=\begin{pmatrix} n -1\\ k \end{pmatrix}+\begin{pmatrix} n -1\\ k-1 \end{pmatrix}$, we have 
\begin{displaymath} \label{technical.5}
\begin{split}
a_l&=(2l+1)\sum_{j'=1}^{l'}\left\{\begin{pmatrix} 2l \\ 2j' \end{pmatrix}-2\begin{pmatrix} 2l \\ 2j'-1 \end{pmatrix}+\begin{pmatrix} 2l\\ 2j'-2 \end{pmatrix}\right\} \\ & 
=(2l+1)\left\{\sum_{j'=1}^{l'}\left\{\begin{pmatrix} 2l-1 \\ 2j' \end{pmatrix}-\begin{pmatrix} 2l-1 \\ 2j'-1 \end{pmatrix}-\begin{pmatrix} 2l-1\\ 2j'-2 \end{pmatrix}\right\} + \sum_{j'=2}^{l'}
\begin{pmatrix} 2l-1 \\ 2j'-3 \end{pmatrix}\right\}\\ & 
=(2l+1)\left\{\begin{pmatrix} 2l-1 \\ l \end{pmatrix}-\begin{pmatrix} 2l-1 \\ l -1\end{pmatrix} -\begin{pmatrix} 2l-1 \\ 0 \end{pmatrix}\right\}
\end{split}
\end{displaymath}
so that $a_l=-(2l+1)$. This finishes to prove formula \eqref{technical.2} in the case where $l$ is even. The case where $l$ is odd follows similarly. 
\end{proof}

\subsection{Commutator estimates}
First, we state the Kato-Ponce commutator estimate \cite{KaPo} (see also Lemma 2.2 in \cite{Po2} for the second estimate) in the case where the functions are defined in $\mathbb R$.
\begin{lemma}[Kato-Ponce commutator estimates] \label{Kato-Ponce}
Let $s \ge 1$, $p,\ p_2, \ p_3 \in (1,\infty)$ and $p_1, \ p_4 \in (1,\infty]$ be such that $\frac1p=\frac1{p_1}+\frac1{p_2}=\frac1{p_3}+\frac1{p_4}$ . Then, 
\begin{equation}  \label{Kato-Ponce1}
\|[J^s,f]g\|_{L^p} \lesssim \|\partial_xf\|_{L^{p_1}}\|J^{s-1}g\|_{L^{p_2}}+\|J^sf\|_{L^{p_3}}\|g\|_{L^{p_4}} \, ,
\end{equation}
and 
\begin{equation}  \label{Kato-Ponce2}
\|J^s(fg)\|_{L^p} \lesssim \|f\|_{L^{p_1}}\|J^{s}g\|_{L^{p_2}}+\|J^sf\|_{L^{p_3}}\|g\|_{L^{p_4}} \, ,
\end{equation}
for any $f, \, g $ defined on $\mathbb R$.
\end{lemma}

The corresponding version of Lemma \ref{Kato-Ponce} on the torus was proved  in Lemma 9.A.1 of \cite{IoKe} (see also estimate (4.8) in \cite{TzVi}).
\begin{lemma} \label{Kato-PonceT}
Let $s \ge 1$. Then, 
\begin{equation} \label{KPT.1}
\|[J^s,f]g\|_{L^2}\lesssim \big(\|f\|_{L^{\infty}}+\|\partial_xf\|_{L^{\infty}}\big)\|g\|_{H^{s-1}}+\|g\|_{L^{\infty}}\|f\|_{H^s} \, ,
\end{equation}
and 
\begin{equation} \label{KPT.100}
\|J^s(fg)\|_{L^2} \lesssim \|f\|_{L^{\infty}}\|J^{s}g\|_{L^{2}}+\|g\|_{L^{\infty}}\|J^sf\|_{L^{2}} \, ,
\end{equation}
for  any $f, \, g$ defined on $\mathbb T$.
\end{lemma}

\begin{remark}The term $\|f\|_{L^{\infty}}$ appearing on the right-hand side of \eqref{KPT.1} is necessary. For example, one could take the  function $f: \mathbb T \rightarrow \mathbb C, x \mapsto \alpha$ for some constant $\alpha$, which satisfies $\partial_xf=0$ but is still in $H^s(\mathbb T)$.
\end{remark}

As a consequence of the estimate \eqref{Kato-Ponce1}, we deduce that 
\begin{equation} \label{Kato-Ponce3}
\big\| J^s(f\partial_xg)-fJ^s\partial_xg\big\|_{L^2} \lesssim 
\|\partial_xf\|_{L^{\infty}}\|J^sg\|_{L^2}+ \|\partial_xg\|_{L^{\infty}}\|J^sf\|_{L^2}\, .
\end{equation}

Here we give a generalized version of estimate \eqref{Kato-Ponce3} with the homogeneous operator $D^s$ instead of the nonhomegenous one $J^s$. The second order case was given by Kwon in Lemma 2.3 of \cite{Kw}. 
\begin{lemma}[Generalized Kato-Ponce commutator estimates] \label{CommEstimates}
Let $s>0$ and $l \in \mathbb N, \ l \ge 2$. Then, 
\begin{equation} \label{CommEstimates.1}
\begin{split}
\big\| D^s(f\partial_x^{2l-1}g)-\sum_{j=0}^{2l-2}&\begin{pmatrix} s \\ j \end{pmatrix}\partial_x^jfD^s\partial_x^{2l-1-j}g\big\|_{L^2} \\ &\lesssim \|\partial_x^{2l-1}f\|_{L^{\infty}}\|D^sg\|_{L^2}+ \|\partial_x^{2l-1}g\|_{L^{\infty}}\|D^sf\|_{L^2}\, ,
\end{split}
\end{equation}
for  any $f, \, g$ defined on $\mathbb R$, and 
\begin{equation} \label{CommEstimates.1T}
\begin{split}
\big\| D^s(f\partial_x^{2l-1}g)-\sum_{j=0}^{2l-2}&\begin{pmatrix} s \\ j \end{pmatrix}\partial_x^jfD^s\partial_x^{2l-1-j}g\big\|_{L^2} \\ &\lesssim \sum_{j=0}^{2l-1}\|\partial_x^jf\|_{L^{\infty}}\|D^sg\|_{L^2}+ \sum_{j=0}^{2l-1}\|\partial_x^jg\|_{L^{\infty}}\|D^sf\|_{L^2}\, ,
\end{split}
\end{equation}
for  any $f, \, g$ defined on $\mathbb T$, where by convention $\begin{pmatrix} s \\ 0 \end{pmatrix}=1$ and  $\begin{pmatrix} s \\ j \end{pmatrix}=\frac{s(s-1)\cdots(s-j+1)}{j!}$ for any $s >0$ and $j \in \mathbb N$ such that $j \le s$.
\end{lemma}

\begin{proof} The proof of estimate \eqref{CommEstimates.1} is an application of the Coifman-Meyer theorem on bilinear Fourier multipliers \cite{CM}. Since it is identical to the proof of Lemma 2.3 in \cite{Kw}, we will omit it.

The proof of estimate \eqref{CommEstimates.1T} is deduced from estimate \eqref{CommEstimates.1} arguing exactly as in the proof of Lemma 9.A.1 in \cite{IoKe}.
\end{proof} 

Finally, we will also need the following commutator estimate involving the high frequencies projection operator $P_{high}$.
\begin{lemma} \label{CommutPhigh}
Let $m \in \mathbb N \cap [1,+\infty)$ and let $P_{high}$ be the operator of projection in high frequencies defined in the introduction. Then, 
\begin{equation} \label{CommutPhigh1}
\big\| [P_{high},f]\partial_x^mg\|_{L^2} \lesssim 
\sum_{j=0}^m\|\partial_x^jf\|_{L^{\infty}}\|g\|_{L^2} \, .
\end{equation}	
\end{lemma}

\begin{proof} We consider the case $M=\mathbb R$. The case $M=\mathbb T$ follows by similar arguments.
Observe by using \eqref{Plow} and integrating by parts that 
\begin{displaymath}
\begin{split}
[P_{high},f]\partial_x^mg&=-P_{low}(f\partial_x^mg)+fP_{low}\partial_x^mg
\\ &=-\gamma\ast(f\partial_x^mg)+f\big(\gamma\ast\partial_x^mg\big)
\\ &=-\sum_{j=0}^m\begin{pmatrix} m \\ j \end{pmatrix}(-1)^{m-j}(\partial_x^{m-j}\gamma )\ast(\partial_x^jfg)+
(-1)^mf(\partial_x^{m}\gamma )\ast g \, ,
\end{split}
\end{displaymath}
where $\gamma=(\eta)^{\vee} \in \mathcal{S}(\mathbb R) \subseteq L^1(\mathbb R)$. Hence estimates \eqref{CommutPhigh1} follows from Young's inequality on convolution.
\end{proof}

\section{Energy estimates} \label{EnergyEst}
in order to simplify the exposition, we will only work with the equation \eqref{homodel}, which is a particular case of \eqref{hodisp}. Note however, as explained in the introduction, that the nonlinear term $u\partial_x^{2l-1}u$ is the most difficult to treat among all the term appearing in the nonlinear term \eqref{nonlinearity}. Moreover, note that for the local theory at this level of regularity, the nonlinearity does not need to be in divergence form.

Finally, we will write the proofs in the case where $M=\mathbb R$. The proofs in the case where $M=\mathbb T$ follow similarly by using \eqref{KPT.1}, \eqref{KPT.100} and \eqref{CommEstimates.1T} instead of \eqref{Kato-Ponce1}, \eqref{Kato-Ponce2} and \eqref{CommEstimates.1T}.

\subsection{Energy estimate for the solution of \eqref{homodel}}
\begin{proposition} \label{EE}
Let $l \in \mathbb N$, $l \ge 2$, $s>s_l=4l-\frac92$ and $u \in C([0,T] : H^s(M))$ be a solution of \eqref{homodel}, with $M=\mathbb R$ or $\mathbb T$. 
Then, we can construct a \emph{modified energy} $E^s(u)$ of the form 
\begin{equation} \label{EE.1b}
E^s(u)(t)=\frac12 \|u(t)\|_{H^s}^2+\sum_{k=3}^{l+1} T_k^s(u)(t) \, ,
\end{equation}
where $T_k^s(u)$ is a term of order $k$ in $u$ and its derivatives, in such a way that the following properties hold true. 

\medskip
(1) \underline{Coercivity}. There exists a positive constant $\delta$ such that  
\begin{equation} \label{EE.4}
\frac14\|u(t)\|_{H^s}^2 \le E^s(u)(t) \le \frac34\|u(t)\|_{H^s}^2 \, ,
\end{equation}
for all $t \in [0,T]$ and for any $u \in C([0,T] : H^s(X))$ such that $\|u\|_{L^{\infty}_TH^s_x}<\delta$.

\medskip
(2) \underline{Energy estimate}.
\begin{equation} \label{EE.5}
\Big|\frac{d}{dt}E^s(u)(t)\Big|  \lesssim \left(\sum_{k=1}^{l-1}\|u(t)\|_{H^s}^k\right)\|u(t)\|_{H^s}^2 \, ,
\end{equation}
for all $t \in (0,T)$ and for any $u \in C([0,T] : H^s(X))$.
\end{proposition}

\begin{proof} Let $u$ be a smooth solution to \eqref{homodel} defined on the time interval $[0,T]$. Following Kwon in \cite{Kw} for the case $l=2$ in $\mathbb R$, we define a modified energy
\begin{equation} \label{EE.3}
E^s(u)(t)=\frac12\|u(t)\|_{H^s}^2+\sum_{k=3}^{l+1}T_k^s(u)(t) \, ,
\end{equation}
where $T_k^s(u)(t)=\int p_k^s(u)(x,t)dx$ and $p_k^s(u)$ is a homogeneous polynomial of degree $k$ in $u$ and its derivatives. We will construct $E^s(u)$ in such a way that \eqref{EE.4} holds true
if $\|u\|_{L^{\infty}H^s}<\delta$ for some small positive number $\delta$, and such that the  energy estimate 
\eqref{EE.5} holds true for all $0 \le t \le T$. 	
	
\medskip	
Now, we explain how to construct the modified energy $E^s(u)$. We first treat the quadratic terms in \eqref{EE.3}. We multiply \eqref{homodel} by $u$ and integrate in space to get that
\begin{equation} \label{EE.6}
\frac12\frac{d}{dt}\|u\|_{L^2}^2 =\int u^2\partial_x^{2l-1}u \le \|\partial_x^{2l-1}u\|_{L^{\infty}}\|u\|_{L^2}^2\, . 
\end{equation}
Next, we apply $D^s$ to \eqref{homodel}, mutiply by $D^su$, integrate in space and use the commutator estimate \eqref{CommEstimates.1} to deduce that
\begin{equation} \label{EE.7}
\begin{split}
\frac12\frac{d}{dt}\|D^su\|_{L^2}^2&=\int D^s(u\partial_x^{2l-1}u)D^su \\ 
&=\mathcal{O}\big(\|\partial_x^{2l-1}u\|_{L^{\infty}}\|D^su\|_{L^2}^2\big)
+\sum_{j=0}^{2l-2}\begin{pmatrix} s \\ j \end{pmatrix}\int \partial_x^juD^s\partial_x^{2l-1-j}uD^su \, .
\end{split}
\end{equation}
To handle the term $\int \partial_x^juD^s\partial_x^{2l-1-j}uD^su$ appearing on the right-hand side of \eqref{EE.7}, we first consider the case where $j$ is even. Let us denote $j=2j'$ with $0\le j' \le l-1$.  By using the notation in \eqref{I.1}, we have
\begin{equation} \label{EE.8}
\int \partial_x^juD^s\partial_x^{2l-j-1}uD^su=\frac12I_{2l-j-1}(\partial_x^ju,D^su,D^su)-\frac12\int\partial_x^{2l-1}u(D^su)^2 \, .
\end{equation}
When, $j=2l-2$, we have that $I_{1}(\partial_x^{2l-2}u,D^su,D^su)=0$. When $0 \le j'\le l-2$, it follows from Lemma \ref{technical} that
\begin{equation} \label{EE.9}
I_{2l-j-1}(\partial_x^ju,D^su,D^su)=\sum_{k=1}^{l-j'-1}\alpha_{k,l-j'-1}\int\partial_x^{2(l-k)-1}u(D^s\partial_x^ku)^2 \, .
\end{equation}
In the case where $j$ is odd, let us denote $j=2j'+1$ with $0 \le j'\le l-2$. We first integrate by parts to obtain 
\begin{equation} \label{EE.10}
\begin{split}
\int \partial_x^juD^s\partial_x^{2l-j-1}uD^su &=-\int \partial_x^{j+1}uD^s\partial_x^{2(l-2-j')+1}uD^su \\ & \quad -\int \partial_x^{j}uD^s\partial_x^{2(l-2-j')+1}uD^s\partial_xu \, .
\end{split}
\end{equation}
By using the notation in \eqref{I.1}, we have that
\begin{equation} \label{EE.11}
\int \partial_x^{j+1}uD^s\partial_x^{2(l-2-j')+1}uD^su=\frac12I_{2l-j-2}(\partial_x^{j+1}u,D^su,D^su)-\int \partial_x^{2l-1}u(D^su)^2
\end{equation}
and the first term on the right-hand side of \eqref{EE.11} can be handled by using Lemma \ref{technical} exactly as in \eqref{EE.9}. To deal with the second term on the right-hand side of \eqref{EE.10}, we have two different cases: or $2(l-2-j')+1=1$ and we are done, or we continue the same process a finite number of times. Finally, we conclude gathering \eqref{EE.6}--\eqref{EE.11} that 
\begin{equation} \label{EE.12}
\frac12\frac{d}{dt}\|u\|_{H^s}^2=\mathcal{O}(\|\partial_x^{l-1}u\|_{L^{\infty}}\|u\|_{H^s}^2)+\sum_{j=1}^{l-1}\beta_j\int \partial_x^{2(l-j)-1}u(D^s\partial_x^ju)^2 \, ,
\end{equation}
where $\beta_j$, $j=1\cdots l-1,$ are real numbers. 
	
The $l-1$ terms appearing in the sum on the right-hand side of \eqref{EE.12} cannot be handled directly. Therefore, we need to add $l-1$ correcting cubic terms to the energy in order to cancel them out. We then define the term $T^s_3(u)$ appearing in \eqref{EE.3} by 
\begin{equation} \label{EE.13}
T^s_3(u)=\sum_{j=0}^{l-2}\gamma_{j}T_{3,j}^s(u) \, ,
\end{equation}
where 
\begin{equation} \label{EE.14}
T_{3,j}^s(u)=\left\{\begin{array}{ll}
\int\partial_{x}^{2j}u(D^{s-2-j}\partial_xu)^2 & \text{if} \ j \ \text{is even} \\ 
\int\partial_{x}^{2j}u(D^{s-1-j}u)^2 & \text{if} \ j \ \text{is odd} 
\end{array} \right.  , \ \text{for} \ 0 \le j \le l-2 ,
\end{equation} 
and $\gamma_{j}$, $0 \le j \le l-2$, are real coefficients to be determined. Note that $D^{s-2-j}$ always makes sense, since $s-2-j \ge 0$ thanks to the hypothesis $s>4l-\frac92$. 
	
We deal for example with the case where $j$ is even. By using the equation \eqref{homodel} and the notation in \eqref{I.1}, we have that
\begin{equation} \label{EE.15}
\frac{d}{dt}T_{3,j}^s(u)=X^s_{3 \to 3,j}(u)+X^s_{3 \to 4,j}(u)
\end{equation}
where $X^s_{3 \to 3,j}(u)$ denotes the cubic terms resulting from $\frac{d}{dt}T_{3,j}^s(u)$ and is defined by
\begin{equation} \label{EE.16}
X^s_{3 \to 3,j}(u)=-I_{2l+1}(\partial_x^{2j}u,D^{s-2-j}\partial_xu,D^{s-2-j}\partial_xu) \, ,
\end{equation}
and $X^s_{3 \to 4,j}(u)$ denotes the fourth-order terms resulting from $\frac{d}{dt}T_{3,j}(u)$ and is defined by
\begin{equation} \label{EE.17}
\begin{split}
X^s_{3 \to 4,j}(u)&=\int \partial_x^{2j}(u\partial_x^{2l-1}u)(D^{s-2-j}\partial_xu)^2\\ & \quad +2\int\partial_x^{2j}uD^{s-2-j}\partial_x(u\partial_x^{2l-1}u)D^{s-2-j}\partial_xu \, .
\end{split}
\end{equation}
We will also denote 
\begin{displaymath}
X^s_{3 \to 3}(u)=\sum_{j=0}^{l-2}\gamma_jX^s_{3 \to 3,j}(u) \quad \text{and} \quad X^s_{3 \to 4}(u)=\sum_{j=0}^{l-2}\gamma_jX^s_{3 \to 4,j}(u) \, .
\end{displaymath}
	
Next, we focus on the cubic terms. It follows from Lemma \ref{technical} that
\begin{equation} \label{EE.18}
X^s_{3 \to 3,j}(u)=-\sum_{k=1}^l\alpha_{k,l}\int\partial_x^{2(l+j-k)+1}u(D^{s-2-j}\partial_x^{1+k}u)^2 \, .
\end{equation}
Since $j$ is even, $j=2j'$, we have $D^{s-2-j}\partial_x^{1+k}u=(-1)^{1+j'}D^s\partial_x^{1+k-(2+j)}u$, so that 
\begin{equation} \label{EE.19}
\begin{split}
X^s_{3 \to 3,j}(u)&=\mathcal{O}(\sum_{k=1}^{j+1}\|\partial_x^{2(l+j-k)+1}u\|_{L^{\infty}}\|u\|_{H^s}^2)\\ & \quad-\sum_{k=1}^{l-(j+1)}\alpha_{k+(j+1),l}\int \partial_x^{2(l-k)-1}u(D^s\partial_x^{k}u)^2 \, .
\end{split}
\end{equation}
We argue the same way when $j$ is odd. Therefore, we deduce from \eqref{EE.13} and \eqref{EE.19} that 
\begin{equation} \label{EE.20}
\begin{split}
X^s_{3 \to 3}(u)&=\mathcal{O}(\sum_{j=0}^{l-2}\sum_{k=1}^{j+1}\|\partial_x^{2(l+j-k)+1}u\|_{L^{\infty}}\|u\|_{H^s}^2)\\ & \quad-\sum_{j=0}^{l-2}\gamma_{j}\sum_{k=1}^{l-(j+1)}\alpha_{k+(j+1),l}\int \partial_x^{2(l-k)-1}u(D^s\partial_x^{k}u)^2 \, .
\end{split}
\end{equation}
We observe from \eqref{EE.12} and \eqref{EE.20} that we can always choose the coefficients $\{\gamma_j\}_{j=0}^{l-2}$ such that 
\begin{equation} \label{EE.21}
\left| \frac12\frac{d}{dt}\|u\|_{H^s}^2+X^s_{3 \to 3}(u) \right| \lesssim M_l(u)\|u\|_{H^s}^2 \, ,
\end{equation}
where 
\begin{equation} \label{Ml}
M_l(u)=\sum_{j=0}^{l-2}\|\partial_x^{2l-1+2j}u\|_{L^{\infty}} \lesssim \|u\|_{H^s}\, 
\end{equation} 
for $s>4l-\frac92$ by the Sobolev embedding.
Indeed, to choose $\{\gamma_{j}\}_{j=0}^{l-2}$ in order to cancel the terms $\beta_j\int \partial_x^{2(l-j)-1}u(D^s\partial_x^{j}u)^2$ for $k=1, \cdots, l-1$, we need to solve a nonhomogeneous $(l-1)\times (l-1)$ linear system. The corresponding linear matrix is triangular with all the coefficients in the diagonal equal to $\alpha_{l,l}$. By using \eqref{technical.2}, we have that $\alpha_{l,l}=(-1)^{l+1}(2l+1) \neq 0$, so that the matrix is invertible.

\medskip
	
Now, we look at the fourth order term  $X^s_{3 \to 4}(u)=\sum_{j=0}^{l-2}\gamma_{j}X^s_{3 \to 4,j}(u)$ where $X^s_{3 \to 4,j}(u)$ is defined in \eqref{EE.17}.  If we are in the case $l=2$, then $X^s_{3 \to 4}(u)$ is easily estimated by using the Kato-Ponce commutator estimate \eqref{Kato-Ponce3}. We briefly explain how to proceed in the case where $l \ge 3$. The first term on the right-hand side of \eqref{EE.17} is clearly bounded by $\|u\|_{H^s}^4$ by using the Sobolev embedding. In the case $j=0$, which is the most difficult, the second one can be rewritten after some integrations by parts as
\begin{equation} \label{EE.22} 
\begin{split}
\int u&D^{s-2}\partial_x(u\partial_x^{2l-1}u)D^{s-2}\partial_xu \\ &= -\int uD^s\big(u\partial_x^{2l-3}u\big)D^su+\int \partial_xuD^s\big(u\partial_x^{2l-3}u\big)D^{s-2}\partial_xu\\ & \quad
-2\int uD^{s-2}\big(\partial_xu\partial_x^{2l-2}u\big)D^su+2\int \partial_xuD^{s-2}\big(\partial_xu\partial_x^{2l-2}u\big)D^{s-2}\partial_xu \\ &
\quad -\int uD^{s-2}\big(\partial_x^2u\partial_x^{2l-3}u\big)D^su+
\int \partial_xuD^{s-2}\big(\partial_x^2u\partial_x^{2l-3}u\big)D^{s-2}\partial_xu \, .
\end{split}
\end{equation}
We explain for example how to handle the first term appearing on the right-hand side of \eqref{EE.22}.
By using \eqref{CommEstimates.1}, we have 
\begin{equation} \label{EE.220}
\begin{split}
\int uD^s\big(u\partial_x^{2l-3}u\big)D^su&=\mathcal{O}(\|u\|_{L^{\infty}}\|\partial_x^{2l-3}u\|_{L^{\infty}}\|u\|_{H^s}^2)\\ & \quad+\sum_{k=0}^{2l-4}\begin{pmatrix} s \\ k \end{pmatrix}\int u\partial_x^kuD^s\partial_x^{2l-3-k}uD^su \, .
\end{split}
\end{equation}
Hence, we obtain arguing as in \eqref{EE.7}--\eqref{EE.12} that 
\begin{equation} \label{EE.23}
\begin{split}
\int uD^s\big(u&\partial_x^{2l-3}u\big)D^su \\&=\mathcal{O}(\|u\|_{H^s}^4)+\sum_{k=0}^{2l-5}\sum_{j=1}^{j_k}\beta_{k,j}\int \partial_x^{2(l-j)-k-3}(u\partial_x^ku)(D^s\partial_x^ju)^2 \, ,
\end{split}
\end{equation}
where $j_k=l-2-[k/2]$, $[k/2]$ denotes the integer part of $k/2$, and $\beta_{k,j}$ are real coefficients for $k=0,\cdots,2l-5$, $ \, j=0,\cdots,j_k$. Thus, for each $k=0,\cdots,2l-5$, we need to add at most  $l-2$ fourth-order terms to the energy in order to cancel out the corresponding sum in $j$ appearing on the right-hand side of \eqref{EE.23}. Of course, we can do the same thing to deal with the other terms appearing on the right-hand side of \eqref{EE.22}. This will define the fourth-order term $T_4^s(u)$. Then, arguing as in \eqref{EE.15}--\eqref{EE.17}, it follows that
\begin{equation} \label{EE.24}
\frac{d}{dt}T_{4}(u)=X^s_{4 \to 4}(u)+X^s_{4 \to 5}(u) \, ,
\end{equation}
where $X^s_{4 \to 4}(u)$ is the fourth-order contribution which will cancel out the problematic terms in  $X^s_{3 \to 4}(u)$ and $X^s_{4 \to 5}(u)$ is the fifth-order contribution. If $l=3$, $X^s_{4 \to 5}(u)$ can be estimated directly by using the Kato-Ponce commutator estimate. If $l \ge 4$, we need to add a fifth-order term to the energy in order to cancel out the problematic terms appearing in $X^s_{4 \to 5}(u)$. 
	
This process will finish after a finite number of modifications to the energy (exactly $l-1$). This yields estimate \eqref{EE.5}, which concludes the proof of Proposition \ref{EE}.

\end{proof}

\begin{remark}
We would like to point out that the method described above is in spirit very similar to the $I$-method (see \cite{CKSTT1,CKSTT},etc).
\end{remark}

\subsection{Energy estimates for the differences} In this subsection, we derive  energy estimates for the difference of two solutions $u_1$ and $u_2$ of \eqref{homodel}. 
\medskip

\begin{proposition} \label{diffEE}
Let $l \in \mathbb N$, $l \ge 2$, $s>s_l=4l-\frac92$ and $u_1, \, u_2 \in C([0,T] : H^s(M))$ be two solutions of \eqref{homodel}, with $M=\mathbb R$ or $\mathbb T$.  We denote $v=u_1-u_2$ the difference between the two solutions, so that $v$ solves 
\begin{equation} \label{diffeq}
\partial_tv+\partial_x^{2l+1}v=v\partial_x^{2l-1}u_1+u_2\partial_x^{2l-1}v \, .
\end{equation}
Then, for $\sigma=0$ or $\sigma=s$, we can construct a \emph{modified energy} $\widetilde{E}^{\sigma}(v)$ of the form 
\begin{equation} \label{diffEE.2}
\widetilde{E}^{\sigma}(v)(t)=\frac12 \|v_{low}(t)\|_{L^2}^2+\frac12 \|D^{\sigma}v_{high}(t)\|_{L^2}^2+\sum_{k=3}^{l+1} \widetilde{T}_k^{\sigma}(u_2,v)(t) \, ,
\end{equation}
where $\widetilde{T}_k^{\sigma}(u_2,v)$ is a term of order $k$ in $u_2$, $v$ and their derivatives, in such a way that the following properties hold true.

\medskip
(1) \underline{Coercivity}. There exists a positive constant $\delta$ such that  
\begin{equation} \label{diffEE.3}
\frac14\|v(t)\|_{H^{\sigma}}^2 \le \widetilde{E}^{\sigma}(v)(t) \le \frac34\|v(t)\|_{H^{\sigma}}^2 \, ,
\end{equation}
for all $t \in [0,T]$ if $\|u_2\|_{L^{\infty}_TH^s_x}<\delta$.

\medskip
(2) \underline{$L^2$-Energy estimate}.
\begin{equation} \label{diffEE.4}
\Big|\frac{d}{dt}\widetilde{E}^0(v)(t)\Big|  \lesssim \left(\sum_{k=1}^{l-1}\Big(\|u_1(t)\|_{H^s}+\|u_2(t)\|_{H^s}\Big)^k\right)\|v(t)\|_{L^2}^2 ,
\end{equation}
for all $t \in (0,T)$ .

\medskip
(3) \underline{$H^s$-Energy estimate}.
\begin{equation} \label{diffEE.5}
\begin{split}
\Big|\frac{d}{dt}\widetilde{E}^s(v)(t) \Big| &\lesssim\left(\sum_{k=1}^{l-1}\Big(\|u_1(t)\|_{H^s}+\|u_2(t)\|_{H^s}\Big)^k\right)\|v(t)\|_{H^s}^2
\\ & \quad+ \left(\sum_{k=1}^{l-1}\|u_2(t)\|_{H^s}^{k-1}\|u_1(t)\|_{H^{s+2l-k}}\right)\|v(t)\|_{L^{\infty}}\|J^sv(t)\|_{L^2} ,
\end{split}
\end{equation}
for all $t \in (0,T)$ .
\end{proposition}

\begin{proof} We begin to estimate the low frequency part of $v$ in $L^2$.	
By \eqref{diffeq}, integration by parts and using the Leibniz rule, we deduce
\begin{displaymath} 
\begin{split}
\frac12\frac{d}{dt}\|v_{low}\|_{L^2}^2 &= \int P_{low}\big(v\partial_x^{2l-1}u_1 \big)v_{low}+\int P_{low}\big(u_2\partial_x^{2l-1}v \big)v_{low} \\ &
=\int P_{low}\big(v\partial_x^{2l-1}u_1 \big)v_{low}-\sum_{j=0}^{2l-1}
\begin{pmatrix} 2l-1 \\ j \end{pmatrix} \int v\partial_x^ju_2\partial_x^{2l-1-j}P_{low}^2v \, ,
\end{split}
\end{displaymath}	
so that 
\begin{equation} \label{diffEE.6}
\frac12\frac{d}{dt}\|v_{low}\|_{L^2}^2\lesssim\left(\|\partial_x^{2l-1}u_1\|_{L^{\infty}}+\sum_{j=0}^{2l-1}\|\partial_x^{j}u_2\|_{L^{\infty}} \right)\|v\|_{L^2}^2 \, .
\end{equation}

Now we turn to the high frequency part of $v$. We get from \eqref{diffeq} and \eqref{CommEstimates.1} that 
\begin{equation} \label{diffEE.7}
\begin{split}
\frac12&\frac{d}{dt}\|D^{\sigma}v_{high}\|_{L^2}^2\\ &=\int D^{\sigma}P_{high}(v\partial_x^{2l-1}u_1)D^{\sigma}v_{high}+\int D^{\sigma}P_{high}(u_2\partial_x^{2l-1}v)D^{\sigma}v_{high} \, .
\end{split}
\end{equation} 

First, we handle the right-hand side of \eqref{diffEE.7}, in the case $\sigma=0$. We get easily that 
\begin{equation} \label{diffEE.8}
\Big|\int P_{high}(v\partial_x^{2l-1}u_1)v_{high}\Big| \lesssim \|\partial_x^{2l-1}u_1\|_{L^{\infty}}\|v\|_{L^2}^2 \, .
\end{equation}
Moreover, 
\begin{displaymath} 
\begin{split}
\int P_{high}(u_2\partial_x^{2l-1}v)v_{high} &=\int u_2\partial_x^{2l-1}v_{high}\, v_{high}
+\int [P_{high},u_2]\partial_x^{2l-1}v \,v_{high}
\\&=\frac12I_{2l-1}(u_2,v_{high},v_{high})+\int [P_{high},u_2]\partial_x^{2l-1}v \,v_{high} \, ,
\end{split}
\end{displaymath}
so that it follows from \eqref{technical.1} and \eqref{CommutPhigh1} that
\begin{equation} \label{diffEE.9}
\begin{split}
\int P_{high}&(u_2\partial_x^{2l-1}v)v_{high}\\&=\mathcal{O}\big(\sum_{j=0}^{2l-1}\|\partial_x^ju_2\|_{L^{\infty}}\|v\|_{L^2}^2 \big)+\sum_{j=1}^{l-1}\widetilde{\beta}_j^0 \int\partial_x^{2(l-j)-1}u_2 (\partial_x^jv_{high})^2 \, ,
\end{split}
\end{equation}
where $\widetilde{\beta}_1^0,\cdots,\widetilde{\beta}_{l-1}^0$ are $l-1$ real numbers. Hence, we conclude from \eqref{diffEE.7}-\eqref{diffEE.9} that 
\begin{equation} \label{diffEE.10}
\begin{split}
\frac12\frac{d}{dt}\|v_{high}\|_{L^2}^2&=\mathcal{O}\Big(\big(\|\partial_x^{2l-1}u_1\|_{L^{\infty}}+\sum_{j=0}^{2l-1}\|\partial_x^ju_2\|_{L^{\infty}}\big)\|v\|_{L^2}^2 \Big)\\ & \quad+\sum_{j=1}^{l-1}\widetilde{\beta}_j^0 \int\partial_x^{2(l-j)-1}u_2 (\partial_x^jv_{high})^2 \, .\end{split}
\end{equation} 

In the case $\sigma=s$, we deduce from \eqref{Kato-Ponce2} that
\begin{equation} \label{diffEE.11}
\begin{split}
\Big|\int D^s&P_{high}(v\partial_x^{2l-1}u_1)D^sv_{high}\Big|\\& \lesssim \|J^s(v\partial_x^{2l-1}u_1)\|_{L^2}\|D^sv\|_{L^2} 
\\ & \lesssim \|\partial_x^{2l-1}J^su_1\|_{L^2}\|v\|_{L^{\infty}}\|J^sv\|_{L^2} +\|\partial_x^{2l-1}u_1\|_{L^{\infty}}\|J^sv\|_{L^2}^2 \, .
\end{split}
\end{equation}
Next, we deal with the second term on the right-hand side of \eqref{diffEE.7}. By using the commutator estimates \eqref{CommutPhigh1} and \eqref{CommEstimates.1}, and arguing exactly as in \eqref{EE.8}-\eqref{EE.12}, we get that 
\begin{equation} \label{diffEE.12}
\begin{split}
\int D^{s}&P_{high}(u_2\partial_x^{2l-1}v)D^{s}v_{high}\\ &=\mathcal{O}\big(\sum_{j=0}^{2l-1}\|\partial_x^ju_2\|_{L^{\infty}}\|D^sv\|_{L^2}^2 +\|D^su_2\|_{L^2}\|\partial_x^{2l-1}v\|_{L^{\infty}}\|D^sv\|_{L^2} \big) \\& \quad+\sum_{j=0}^{2l-2}\begin{pmatrix} s \\ j \end{pmatrix}\int \partial_x^ju_2D^s\partial_x^{2l-1-j}v_{high}D^sv_{high} 
\\ & =\mathcal{O}\big(\sum_{j=0}^{2l-1}\|\partial_x^ju_2\|_{L^{\infty}}\|D^sv\|_{L^2}^2 +\|D^su_2\|_{L^2}\|\partial_x^{2l-1}v\|_{L^{\infty}}\|D^sv\|_{L^2} \big) \\& \quad+
\sum_{j=1}^{l-1}\widetilde{\beta}_j^s\int\partial_x^{2(l-j)-1}u_2 (D^s\partial_x^jv_{high})^2 \, ,
\end{split}
\end{equation}
where $\widetilde{\beta}_1^s,\cdots,\widetilde{\beta}_{l-1}^s$ are $l-1$ real numbers. Therefore, we deduce gathering \eqref{diffEE.11}-\eqref{diffEE.12} that
\begin{equation} \label{diffEE.13}
\begin{split}
\frac12\frac{d}{dt}&\|D^sv_{high}\|_{L^2}^2\\&=\mathcal{O}\Big(\big(\|\partial_x^{2l-1}u_1\|_{L^{\infty}}+\sum_{j=0}^{2l-1}\|\partial_x^ju_2\|_{L^{\infty}}\big)\|v\|_{H^s}^2 \Big)
\\& +\mathcal{O}\Big(\|\partial_x^{2l-1}J^su_1\|_{L^2}\|v\|_{L^{\infty}}\|J^sv\|_{L^2} 
+\|D^su_2\|_{L^2}\|\partial_x^{2l-1}v\|_{L^{\infty}}\|D^sv\|_{L^2}\Big)
\\ & \quad+\sum_{j=1}^{l-1}\widetilde{\beta}_j^s \int\partial_x^{2(l-j)-1}u_2 (D^s\partial_x^jv_{high})^2 \, .\end{split} 
\end{equation}

Observe that in both cases $\sigma=0$ corresponding to estimate \eqref{diffEE.10} and $\sigma=s$ corresponding to estimate \eqref{diffEE.13}, we cannot handle directly by integration by parts the $l-1$ third-order terms $\sum_{j=1}^{l-1}\widetilde{\beta}_j^{\sigma} \int\partial_x^{2(l-j)-1}u_2 (D^{\sigma}\partial_x^jv_{high})^2 $ appearing on the right-hand side of \eqref{diffEE.10} and \eqref{diffEE.13}. Therefore, we need to add $l-1$ correcting cubic terms to the energy in order to cancel them out. We then define $\widetilde{T}_3^{\sigma}(u_2,v)$ appearing in \eqref{diffEE.2} by 
\begin{equation} \label{diffEE.14}
\widetilde{T}^{\sigma}_3(u_2,v)=\sum_{j=0}^{l-2}\widetilde{\gamma}_{j}\widetilde{T}_{3,j}^{\sigma}(u_2,v) \, ,
\end{equation}
where 
\begin{equation} \label{diffEE.15}
\widetilde{T}_{3,j}^{\sigma}(u_2,v)=\left\{\begin{array}{ll}
\int\partial_{x}^{2j}u_2(D^{{\sigma}-2-j}\partial_xv_{high})^2 & \text{if} \ j \ \text{is even} \\ 
\int\partial_{x}^{2j}u_2(D^{{\sigma}-1-j}v_{high})^2 & \text{if} \ j \ \text{is odd} 
\end{array} \right.  , \ \text{for} \ 0 \le j \le l-2 ,
\end{equation} 
and $\widetilde{\gamma}_{j}$, $0 \le j \le l-2$, are real coefficients to be determined. Note that $D^{\sigma-2-j}\partial_xv_{high}$ and $D^{\sigma-1-j}v_{high}$ make sense even in the case $\sigma=0$, since $v_{high}=P_{high}v$ is the projection of $v$ in high frequencies.

Arguing exactly as in \eqref{EE.18}-\eqref{EE.21}, we find that 
\begin{equation} \label{diffEE.16}
\frac{d}{dt}\widetilde{T}_{3,j}^{\sigma}(u_2,v)=\widetilde{X}^{\sigma}_{3 \to 3,j}(u_2,v)+\widetilde{X}^{\sigma}_{3 \to 4,j}(u_1,u_2,v)
\end{equation}
where $\widetilde{X}^{\sigma}_{3 \to 3,j}(u_2,v)$ denotes the cubic terms resulting from $\frac{d}{dt}\widetilde{T}_{3,j}^{\sigma}(u_2,v)$ and is defined by
\begin{equation} \label{diffEE.17}
\widetilde{X}^{\sigma}_{3 \to 3,j}(u_2,v)=-I_{2l+1}(\partial_x^{2j}u_2,D^{\sigma-2-j}\partial_xv,D^{\sigma-2-j}\partial_xv) \, ,
\end{equation}
and $\widetilde{X}^{\sigma}_{3 \to 4,j}(u_1,u_2,v)$ denotes the fourth-order terms resulting from $\frac{d}{dt}\widetilde{T}_{3,j}^{\sigma}(u_2,v)$ and is defined by
\begin{equation} \label{diffEE.18}
\begin{split}
\widetilde{X}^{\sigma}_{3 \to 4,j}(u_1,u_2,v)&=\int \partial_x^{2j}(u_2\partial_x^{2l-1}u_2)(D^{\sigma-2-j}\partial_xv_{high})^2\\ & \quad +2\int\partial_x^{2j}u_2D^{\sigma-2-j}\partial_xP_{high}(v\partial_x^{2l-1}u_1)D^{\sigma-2-j}\partial_xv_{high} 
\\ &\quad +2\int\partial_x^{2j}u_2D^{\sigma-2-j}\partial_xP_{high}(u_2\partial_x^{2l-1}v)D^{\sigma-2-j}\partial_xv_{high} 
\\ & =: I_j^{\sigma}(u_2,v)+II^{\sigma}_j(u_1,u_2,v)+III^{\sigma}_j(u_2,v) \, .
\end{split}
\end{equation} 

Following the lines of \eqref{EE.18}--\eqref{EE.21} and using respectively \eqref{diffEE.10} and \eqref{diffEE.13} and the Sobolev embedding for $s>4l-\frac92$, we deduce that 
\begin{equation} \label{diffEE.19}
\begin{split}
\Big| \frac12\frac{d}{dt}\|v_{high}\|_{L^2}^2+\widetilde{X}^0_{3 \to 3}(u_2,v) \Big|&\lesssim \Big(\|\partial_x^{2l-1}u_1\|_{L^{\infty}}+\sum_{j=0}^{4l-5}\|\partial_x^ju_2\|_{L^{\infty}}\Big)\|v\|_{L^2}^2 \, ,\\ & \lesssim \big(\|u_1\|_{H^s}+\|u_2\|_{H^s}\big)\|v\|_{L^2}^2
\end{split}
\end{equation}
and 
\begin{equation} \label{diffEE.20}
\begin{split}
\Big| \frac12\frac{d}{dt}\|D^sv_{high}\|_{L^2}^2+\widetilde{X}^s_{3 \to 3}(u_2,v) \Big| &\lesssim  \big(\|u_1\|_{H^s}+\|u_2\|_{H^s}\big)\|v\|_{H^s}^2
\\ & \quad +\|u_1\|_{H^{s+2l-1}}\|v\|_{L^{\infty}}\|J^sv\|_{L^2} \, ,
\end{split}
\end{equation}
if the $l-1$ coefficients $\widetilde{\gamma}_0,\cdots,\widetilde{\gamma}_{l-2}$ are chosen correctly.

Now, we explain how to deal with $\widetilde{X}^{\sigma}_{3 \to 4,j}(u_1,u_2,v)$ for $j=0,\cdots,l-2$. We need to deal with the three terms $I^{\sigma}_j(u_2,v)$, $II^{\sigma}_j(u_1,u_2,v)$ and $III^{\sigma}_j(u_2,v)$ on the right-hand side of \eqref{diffEE.18}. By using H\"older's inequality and the Sobolev embedding, we easily get in both cases $\sigma=0$ and $\sigma=s$ that
\begin{displaymath} 
\big| I^{\sigma}_j(u_2,v) \big| \lesssim \|u_2\|_{H^s}^2\|v\|_{H^{\sigma}}^2 \, .
\end{displaymath}
In the case $\sigma=0$, we estimate $II$ similarly and get 
\begin{displaymath} 
\big| II^0_j(u_1,u_2,v) \big| \lesssim \|u_1\|_{H^s}\|u_2\|_{H^s}\|v\|_{L^2}^2 \, .
\end{displaymath}
In the case $\sigma=s$, we use the commutator estimate \eqref{Kato-Ponce2} and argue as in \eqref{diffEE.11} to deduce that 
\begin{displaymath} 
\big| II^s_j(u_1,u_2,v) \big| \lesssim \|u_1\|_{H^s}\|u_2\|_{H^s}\|v\|_{H^s}^2 +\|u_2\|_{H^s}\|u_1\|_{H^{s+2l-2}}\|v\|_{L^{\infty}}\|v\|_{H^s}\, .
\end{displaymath}
Finally, in order to control $III^{\sigma}(u_2,v)=\sum_{j=0}^{l-2}\widetilde{\gamma}_jIII^{\sigma}_j(u_2,v)$, we follow the argument in \eqref{EE.22}-\eqref{EE.24}. In the case $l=2$, it suffices to use the Kato-Ponce commutator estimate. In the case $l \ge 3$, we need to introduce a fourth-order modification to the energy $\widetilde{T}_4^{\sigma}(u_2,v)$ in such a way that 
\begin{displaymath}
\frac{d}{dt}\widetilde{T}_{4}^{\sigma}(u_2,v)=\widetilde{X}^{\sigma}_{4 \to 4}(u_2,v)+\widetilde{X}^{\sigma}_{4 \to 5}(u_1,u_2,v) 
\end{displaymath}
and 
\begin{displaymath}
\Big| III^{\sigma}(u_2,v)+\widetilde{X}^{\sigma}_{4 \to 4}(u_2,v) \Big| \lesssim  \|u_2\|_{H^s}^2\|v\|_{H^{\sigma}}^2 \, .
\end{displaymath}
Here, $\widetilde{X}^{\sigma}_{4 \to 4}(u_2,v)$, respectively $\widetilde{X}^{\sigma}_{4 \to 5}(u_1,u_2,v) $, denotes the fourth-order terms, respectively fifth-order terms, coming from $\frac{d}{dt}\widetilde{T}_{4}^{\sigma}(u_2,v)$.

In the case where $l=3$, $\widetilde{X}^{\sigma}_{4 \to 5}(u_1,u_2,v)$ can be estimated directly by using the Kato-Ponce commutator estimate. If $l \ge 4$, we need to add a fifth-order term to the energy in order to cancel out the problematic terms appearing in $\widetilde{X}^{\sigma}_{4 \to 5}(u_1,u_2,v)$.

This process will finish after a finite number of modifications to the energy (exactly $l-1$). This yields the proofs of estimates \eqref{diffEE.4} and \eqref{diffEE.5}, which concludes the proof of Proposition \ref{diffEE}.
\end{proof}

\section{Proof of Theorem \ref{localtheory}} \label{ProofTheo}
As mentioned in the previous section, we will prove Theorem \ref{localtheory} in the particular case of equation \eqref{homodel} for the sake of clarity. In this section, we fix $l \in \mathbb N$, $l \ge 2$ and work with $s>s_l=4l-\frac92$.
\smallskip

By scaling, it is enough to deal with initial data $u(\cdot,0)=u_0$ having small $H^s$-norm. Indeed of $u$ is a solution to \eqref{homodel} defined on a time interval $[0,T]$, for some positive time $T$, then, for all $\lambda>0$, $u_{\lambda}(x,t)=\lambda^2u(\lambda x,\lambda^{2l+1}t)$ is also a solution to \eqref{homodel} defined on a time interval $[0,T/\lambda^{2l+1}]$. For any $\delta>0$, we define $\mathcal{B}^s(\delta)$ the ball of $H^s(M)$ centered at the origin and of radius $\delta$. Since $$\|u_{\lambda}(\cdot,0)\|_{H^s} \lesssim \lambda^{\frac32} (1+\lambda^s)\|u_0\|_{H^s} \, ,$$  we can force $u_{\lambda}(\cdot,0)$ to belong to $\mathcal{B}^s(\delta)$ by choosing $\lambda \sim \min\{\big(\frac{\delta}{\|u_0\|_{H^s}}\big)^{\frac23},1\}$. Therefore, the existence and uniqueness of a solution to \eqref{homodel} on a time interval $[0,1]$ for small initial data $\|u_0\|_{H^s}$ will ensure the existence and uniqueness of a solution to \eqref{homodel} for arbitrarily large initial data on a time interval $[0,T]$ for $T \sim \min\{ \|u_0\|_{H^s}^{-2(2l+1)/3},1\}$.

From now on, we assume that $u_0 \in H^s(M)$ satisfies $\|u_0\|_{H^s} \le \delta$, where $\delta$ is a small positive number which will be fixed later. 
\smallskip

The proof of Theorem \ref{localtheory} is based on parabolic regularization, energy estimates and the Bona-Smith argument. For $\mu>0$, we consider the regularized problem 
\begin{equation} \label{dissiphomodel}
\left\{
\begin{array}{l}\partial_tu^{\mu}+\partial_x^{2l+1}u^{\mu}+\mu(-1)^{l+1} \partial_x^{2l+2}u^{\mu}=u^{\mu}\partial_x^{2l-1}u^{\mu} \\ u^{\mu}(\cdot,0)=u_0 \, ,
\end{array} \right. 
\end{equation}

 \smallskip
 Combining the arguments of Lemma 2 in \cite{Sa} and Theorem 5.14 in \cite{IoIo}, we obtain a local well-posedness result for the IVP \eqref{dissiphomodel}.
 \begin{proposition} \label{parabolic}
 Let $s>s_l=4l-\frac92$.
 For every $\mu>0$ and every $u_0 \in H^s(M)$, with $M=\mathbb R$ or $\mathbb T$, there exist a positive maximal time of existence $T_{\mu}=T^s_{\mu}(u_0)$ and a unique solution $u^{\mu}$ to \eqref{homodel} in $C([0,T_{\mu}) : H^s(M))$.
 
Moreover, the \lq\lq extension principle\rq\rq \, holds, \textit{i.e.}: 
\begin{equation} \label{parabolic.1}
\text{either} \quad T_{\mu}=+\infty \quad \text{or} \quad \limsup_{t \nearrow T_{\mu}} \|u^{\mu}(t)\|_{H^s}=+\infty
\end{equation} 
and, for very $0<T<T_{\mu}$, the flow map data-solution 
$$:v_0 \in H^s(M) \mapsto v \in C([0,T] : H^s(M) )$$  is continuous in a neighborhood of $u_0$ in $H^s(M)$. 
\end{proposition}
 
\subsection{\textit{A priori} estimates on the solutions $u^{\mu}$} \label{APE}
 
\begin{proposition} \label{apriorienergy}
Assume that $s>s_l=4l-\frac92$ and $0<\mu \le 1$. Let $u^{\mu} \in C([0,T_{\mu}) : H^s(M))$ be the solution of \eqref{dissiphomodel} obtained in Proposition \ref{parabolic} and $E^s(u^{\mu})$ be the modified energy constructed in Proposition \ref{EE}. Then, there exists a positive constant $\delta_0$ (independent of $\mu \in (0,1]$)  such that the following properties hold true. 
 
\medskip
(1) \underline{Coercivity}. 
\begin{equation} \label{apriorienergy.1}
\frac14\|u^{\mu}(t)\|_{H^s}^2 \le E^s(u^{\mu})(t) \le \frac34\|u^{\mu}(t)\|_{H^s}^2 \, ,
\end{equation}
for all $t \in [0,T]$, if $u^{\mu}$ satisfies $\|u^{\mu}\|_{L^{\infty}_TH^s_x} \le \delta_0$ for some $0<T<T_{\mu}$.
 
\medskip
(2) \underline{Energy estimate}.
\begin{equation} \label{apriorienergy.2}
\left|\frac{d}{dt}E^s(u^{\mu})(t)+\mu \|\partial_x^{l+1}u^{\mu}(\cdot,t)\|_{H^s}^2 \right|\lesssim \left(\sum_{k=1}^{l-1}\|u^{\mu}(t)\|_{H^s}^k\right)\|u^{\mu}(t)\|_{H^s}^2 \, ,
\end{equation}
for all $t \in (0,T)$, if $u^{\mu}$ satisfies $\|u^{\mu}\|_{L^{\infty}_TH^s_x} \le \delta_0$ for some $0<T<T_{\mu}$.	
\end{proposition}

\begin{proof} The proof of Proposition \ref{apriorienergy} follows the lines of the one of Proposition \ref{EE} for the dissipationless equation. We explain now how to deal with the dissipation in the argument. 

Arguing exactly as in \eqref{EE.6}-\eqref{EE.12} with the solutions $u^{\mu}$ of \eqref{dissiphomodel}, we get that
\begin{equation} \label{apriorienergy.3}
\begin{split}
\frac12\frac{d}{dt}\|u^{\mu}\|_{H^s}^2&+\mu \int(\partial_x^{l+1}u^{\mu})^2+\mu \int(D^s\partial_x^{l+1}u^{\mu})^2 
\\ & =\mathcal{O}(\|\partial_x^{l-1}u^{\mu}\|_{L^{\infty}}\|u^{\mu}\|_{H^s}^2)+\sum_{j=1}^{l-1}\beta_j\int \partial_x^{2(l-j)-1}u^{\mu}(D^s\partial_x^ju^{\mu})^2 \, ,
\end{split}
\end{equation}
where $\beta_j$, $j=1\cdots l-1,$ are real numbers. 

In order to handle the $l-1$ cubic terms on the right-hand side of \eqref{apriorienergy.3}, we introduce the cubic modified energy $T_3^s(u^{\mu})=\sum_{j=0}^{l-2}\gamma_{j}T_{3,j}^s(u^{\mu}) $
where 
\begin{displaymath} 
T_{3,j}^s(u^{\mu})=\left\{\begin{array}{ll}
\int\partial_{x}^{2j}u^{\mu}(D^{s-2-j}\partial_xu^{\mu})^2 & \text{if} \ j \ \text{is even} \\ 
\int\partial_{x}^{2j}u^{\mu}(D^{s-1-j}u^{\mu})^2 & \text{if} \ j \ \text{is odd} 
\end{array} \right.  , \ \text{for} \ 0 \le j \le l-2 ,
\end{displaymath} 
and $\gamma_{j}$, $0 \le j \le l-2$, are real numbers. By using the equation in \eqref{dissiphomodel}, we have that 
\begin{equation} \label{apriorienergy.4}
\frac{d}{dt}T_{3}^s(u^{\mu})=X^s_{3 \to 3}(u^{\mu})+\mu D_{3 \to 3}(u^{\mu})+X^s_{3 \to 4}(u^{\mu})
\end{equation}
where $X^s_{3 \to 3}(u^{\mu})$ denotes the cubic terms resulting from $\frac{d}{dt}T_{3}^s(u^{\mu})$ associated to the dispersive term $\partial_x^{2l+1}u^{\mu}$, $ D_{3 \to 3}(u^{\mu})$ denotes the cubic terms  resulting from $\frac{d}{dt}T_{3}^s(u^{\mu})$ associated to the dissipative term $(-1)^{l+1}\partial_x^{2l+2}u^{\mu}$ and $X^s_{3 \to 4}(u^{\mu})$ denotes the fourth-order terms resulting from $\frac{d}{dt}T_{3}(u^{\mu})$. 

Arguing as in \eqref{EE.18}-\eqref{EE.21}, we can choose the $l-1$ coefficients $\gamma_0,\cdots,\gamma_{l-2}$ so that 
\begin{equation} \label{apriorienergy.5}
 \left|\sum_{j=1}^{l-1}\beta_j\int \partial_x^{2(l-j)-1}u^{\mu}(D^s\partial_x^ju^{\mu})^2+X^s_{3 \to 3}(u) \right|\lesssim \|u^{\mu}\|_{H^s}^2 \, ,
\end{equation}
since $s>4l-\frac92$.

Next, we explain how to handle the term $D_{3\to3}^s(u^{\mu})=\sum_{j=0}^{l-2}D_{3\to3,j}^s(u^{\mu})$, where $D_{3\to3,j}^s(u^{\mu})$ denotes the cubic terms resulting from $\frac{d}{dt}T_{3,j}^s(u^{\mu})$ associated to the dispersive term $\partial_x^{2l+1}u^{\mu}$. We treat for example the case where $j$ is even. We deduce, integrating by parts and using the notation in \eqref{I.1}, that
\begin{displaymath}
\begin{split}
&(-1)^lD_{3 \to 3,j}^s(u^{\mu})\\&=\int \partial_x^{2j+2l+2}u^{\mu}(D^{s-2-j}\partial_xu^{\mu})^2+2\int \partial_x^{2j}u^{\mu}D^{s-2-j}\partial_x^{2l+3}u^{\mu}D^{s-2-j}\partial_xu^{\mu}
\\ &=-I_{2l+1}(\partial_x^{2j+1}u^{\mu},D^{s-2-j}\partial_xu^{\mu},D^{s-2-j}\partial_xu^{\mu})
-2\int \partial_x^{2j}u^{\mu}D^{s-2-j}\partial_x^{2l+2}u^{\mu}D^{s-2-j}\partial_x^2u^{\mu}.
\end{split}
\end{displaymath}
Then, we get, after repeating $l-1$ times this operation that 
\begin{align*} 
D_{3 \to 3,j}^s(u^{\mu})&=\sum_{k=0}^{l-1}(-1)^{l+k+1}I_{2(l-k)+1}(\partial_x^{2j+1}u^{\mu},D^{s-2-j}\partial_x^{k+1}u^{\mu},D^{s-2-j}\partial_x^{k+1}u^{\mu}) 
\nonumber\\ & \quad +\sum_{k=1}^{l-1}(-1)^{l+k+1}\int \partial_x^{2j+2(l+1-k)}u^{\mu} (D^{s-2-j}\partial_x^{k+1}u^{\mu})^2
\\ & \quad  - 2\int \partial_x^{2j}u^{\mu}(D^{s-2-j}\partial_x^{l+2}u^{\mu})^2 \, . \nonumber
\end{align*}
Moreover, it follows from Lemma \ref{technical} that 
\begin{displaymath} 
\begin{split}
I_{2(l-k)+1}(\partial_x^{2j+1}&u^{\mu},D^{s-2-j}\partial_x^{k+1}u^{\mu},D^{s-2-j}\partial_x^{k+1}u^{\mu}) 
\\ &=\sum_{m=1}^{l-k}\alpha_{m,l-k}\int \partial_x^{2(l-k+m)+2j+1}u^{\mu}(D^{s-2-j}\partial_x^{k+1+m}u^{\mu})^2
\end{split}
\end{displaymath}
for each $0 \le k \le l-1$, where $\alpha_{1,l-k},\cdots,\alpha_{m,l-k}$ are $l-k$ real numbers. Hence, we deduce from the Sobolev embedding and the smallness assumption $\|u^{\mu}\|_{L^{\infty}_TH^s_x} \le \delta_0$ that
\begin{equation} \label{apriorienergy.6}
\left|\mu D_{3\to3}^s(u^{\mu})+\mu \int(\partial_x^{l+1}u^{\mu})^2+\mu \int(D^s\partial_x^{l+1}u^{\mu})^2 \right| \lesssim \mu \|u^{\mu}\|_{H^s}^3 \, ,
\end{equation}
if $\delta_0$ is chosen to be small enough.

Therefore, we conclude gathering \eqref{apriorienergy.3}, \eqref{apriorienergy.5} and \eqref{apriorienergy.6} that 
\begin{displaymath} 
\left|\frac12\frac{d}{dt}\|u^{\mu}\|_{H^s}^2+ X_{3\to3}^s(u^{\mu})+\mu D_{3\to3}^s(u^{\mu})+\mu \int(\partial_x^{l+1}u^{\mu})^2+\mu \int(D^s\partial_x^{l+1}u^{\mu})^2 \right| \lesssim \|u^{\mu}\|_{H^s}^3
\end{displaymath}

In view of \eqref{apriorienergy.4}, it remains to control the fourth-order term $X_{3\to4}^s(u^{\mu})$. We proceed as at the end of the proof of Proposition \ref{EE}. If $l=2$, it can be done directly by using the Kato-Ponce commutator estimate. If $l=2$, we need to add a fourth-order contribution $T_4^s(u^{\mu})$ to the energy in order to cancel out the bad terms appearing in $X_{3\to4}^s(u^{\mu})$. When differentiating $T_4^s(u^{\mu})$ with respect to the time, we get that 
\begin{displaymath}
\frac{d}{dt}T_4^s(u^{\mu})=X^s_{4 \to 4}(u^{\mu})+\mu D^s_{4 \to 4}(u^{\mu})+X^s_{4 \to 5}(u^{\mu}) \, .
\end{displaymath} 
We estimate the dissipative contribution $ \mu D^s_{4 \to 4}(u^{\mu})$ exactly as we did for $ \mu D^s_{3\to 3}(u^{\mu})$\footnote{It is actually easier since $D^s_{4 \to 4}(u^{\mu})$ contains fewer derivatives than $D^s_{3\to 3}(u^{\mu})$.} In the case $l=3$, we can estimate $X^s_{4 \to 5}(u^{\mu})$ by using the Kato-Ponce commutator estimate. In the case where $l \ge 4$, we need to repeat the process one more step. 

This process will finish after a finite number of modifications to the energy (exactly $l-1$). This yields estimate \eqref{apriorienergy.2}, which concludes the proof of Proposition \ref{apriorienergy}.
\end{proof}

With Proposition \ref{apriorienergy} in hand, we are in a position to derive suitable  \textit{a priori} estimates on the solutions $u^{\mu}$ of \eqref{dissiphomodel} at the $H^s$-level.  

\begin{lemma} \label{aprioriestimate}
Assume that $s>s_l=4l-\frac92$. There exists $\delta_1>0$ such that if $u_0 \in H^s(M)$ satisfies $\|u_0\|_{H^s} \le \delta_1$, then the solution $u^{\mu} \in C([0,T_{\mu}) : H^s(M))$ to \eqref{dissiphomodel} is defined on a maximal time of existence $T_{\mu} \ge 1$ and satisfies
\begin{equation} \label{aprioriestimate.1}
\|u^{\mu}\|_{L^{\infty}_1H^s_x}+\mu\|\partial_x^{l+1}u^{\mu}\|_{L^2_1H^s_x} \lesssim \|u_0\|_{H^s} \, ,
\end{equation}
for all $0<\mu \le 1$.
\end{lemma}

\begin{proof} Fix $\mu \in (0,1]$ and let $u^{\mu}$ be the solution of \eqref{dissiphomodel} defined on its maximal time interval $[0,T_{\mu})$.
	
Assume by contradiction that $T_{\mu} <1$. Fix some $T \in (0,T_{\mu})$. By integrating \eqref{apriorienergy.2} and using \eqref{apriorienergy.1}, we deduce that 
\begin{displaymath}
\|u^{\mu}\|_{L^{\infty}_TH^s_x}^2 +\mu\int_0^T\|\partial_x^{l+1}u^{\mu}(\cdot,t)\|_{H^s}^2dt \le \|u_0^{\mu}\|_{H^s_x}^2+C\sum_{k=1}^{l-1}\|u^{\mu}\|_{L^{\infty}_TH^s_x}^{k+2} \, ,
\end{displaymath}
as soon as $\|u^{\mu}\|_{L^{\infty}_TH^s_x} \le \delta_0$. Moreover, it follows from Proposition \ref{parabolic} that 
\begin{displaymath}
\lim_{t \to 0}\|u^{\mu}(t)\|_{H^s}= \|u_0\|_{H^s} \, .
\end{displaymath} 
By using a continuity argument, these two facts ensure the existence of a small positive constant $\widetilde{\delta_1}$ (independent of $\delta_0$) such that if $\|u_0\|_{H^s} \le \widetilde{\delta_1}$, then 
\begin{displaymath}
\|u^{\mu}\|_{L^{\infty}_TH^s_x} \le C\|u_0\|_{H^s} \, ,
\end{displaymath}  
as soon as $\|u^{\mu}\|_{L^{\infty}_TH^s_x} \le \delta_0$. Therefore, if $\|u_0\|_{H^s} \le \delta_1:=\min\{\widetilde{\delta_1},\delta_0/2C\}$, the solution $u^{\mu}$ of \eqref{dissiphomodel} satisfies 
\begin{displaymath}
\|u^{\mu}\|_{L^{\infty}_TH^s_x} \le \delta_0/2 \, .
\end{displaymath}

This implies that $\limsup_{t \nearrow T_{\mu}} \|u^{\mu}(t)\|_{H^s} \le \delta_0/2$, since $T$ was chosen arbitrarily in $(0,T_{\mu})$. Hence, it follows from the \lq\lq extension principle\rq\rq \,  in \eqref{parabolic.1} that $T_{\mu}=+\infty$, which is absurd. 

Therefore, we deduce that $T_{\mu} \ge 1$, and then \eqref{aprioriestimate.1} follows by reapplying the above argument with $T=1$.
\end{proof}
 
\subsection{Existence}  \label{EX} Let $u_0 \in H^s(M)$. As explained above, we can always assume that $\|u_0\|_{H^s} \le \delta_1 $. Then, it follows from Lemma \ref{aprioriestimate}, that, for $0<\mu \le 1$, the solutions $u^{\mu}$ obtained in Proposition \ref{parabolic} are defined on time interval $[0,1]$ and satisfy the \textit{a priori}  estimate \eqref{aprioriestimate.1}. 

First, we will prove that $\{u^{\mu}\}_{0<\mu \le 1}$ is a Cauchy sequence in $C([0,1] : H^{s_-}(M))$, where $s_-$ is any number slightly lesser than $s$. Let $0<\mu'<\mu \le 1$. We define $v=u^{\mu}-u^{\mu'}$. Then $v$ is solution to the equation
 \begin{equation} \label{dissipdiff}
\begin{split}
 \partial_tv+\partial_x^{2l+1}v+\mu(-1)^{l+1}\partial_x^{2l+2}v+(\mu-\mu')(-1)^{l+1}&\partial_x^{2l+2}u^{\mu'}
 \\ &=v\partial_x^{2l-1}u^{\mu}+u^{\mu'}\partial_x^{2l-1}v 
\end{split}
 \end{equation}
with initial datum $v(\cdot,0)=0$. In the next Proposition, we derive a $L^2$ energy estimate for $v$.

\begin{proposition} \label{diffapriorienergy}
Assume that $s>s_l=4l-\frac92$ and $0<\mu'<\mu \le 1$. Let $v \in C([0,1] : H^s(M))$ be the solution of \eqref{dissipdiff}  and let $\tilde{E}^0(v)$ be the modified energy constructed in Proposition \ref{diffEE}. Then, there exists a small positive constant $\delta_2$ (independent of $0<\mu'<\mu \le 1$)  such that the following properties hold true. 

\medskip
(1) \underline{Coercivity}. 
\begin{equation} \label{diffapriorienergy.1}
\frac14\|v(t)\|_{L^2}^2 \le \widetilde{E}^0(v)(t) \le \frac34\|v(t)\|_{L^2}^2 \, ,
\end{equation}
for all $t \in [0,1]$ if $\|u^{\mu'}\|_{L^{\infty}_1H^s_x}<\delta_2$.

\medskip
(2) \underline{$L^2$-Energy estimate}.
\begin{equation} \label{diffapriorienergy.2}
\begin{split}
&\left|\frac{d}{dt}\widetilde{E}^0(v)(t)+\mu \|\partial_x^{l+1}v(\cdot,t)\|_{L^2}^2\right|  \\ &\lesssim \left(\sum_{k=1}^{l-1}\Big(\|u^{\mu}(t)\|_{H^s}+\|u^{\mu'}(t)\|_{H^s}\Big)^k\right)\|v(t)\|_{L^2}^2+(\mu-\mu')\left|\int \partial_x^{l+1}u^{\mu'}\partial_x^{l+1}v\right| ,
\end{split}
\end{equation}
for all for all $t \in (0,1)$ if $\|u^{\mu'}\|_{L^{\infty}_1H^s_x} < \delta_2$.
\end{proposition}

\begin{proof} The proof of Proposition \ref{diffapriorienergy} follows the lines of the proof of Proposition \ref{diffEE} for the dissipationless equation. 
	
Note that this time, we can control the terms resulting from the dissipation when deriving the higher order terms in the modified energy just by using the Sobolev embedding, since we are at the $L^2$-level and $u^{\mu}$ is bounded in $H^s$ for $s>4l-\frac92$.

Moreover, the last term appearing on the right-hand side of \eqref{diffapriorienergy.2} corresponds to the contribution of the last term on the left-hand side of \eqref{dissipdiff}.
\end{proof}
 
According to \eqref{aprioriestimate.1}, there exists a small positive number $0<\delta_3\le \delta_1$, such that if $\|u_0\|_{H^s} \le \delta < \delta_3$, then $\|u^{\mu}\|_{L^{\infty}_1H^s_x} < \delta_2$. Thus, it follows from \eqref{diffapriorienergy.1}, \eqref{diffapriorienergy.2} and the \textit{a priori} estimate \eqref{aprioriestimate.1} that 
\begin{displaymath}
\frac{d}{dt}\widetilde{E}^0(v)(t) \le C \sum_{k=1}^{l-1}\delta^k \widetilde{E}^0(v)(t)+C(\mu-\mu')\delta^2 \, .
\end{displaymath}
Hence, we deduce by using Gronwall's inequality and \eqref{diffapriorienergy.1} that
\begin{displaymath}
\|u^{\mu}-u^{\mu'}\|_{L^{\infty}_1L^2_x} \lesssim \mu-\mu' \underset{\mu',\mu \to 0}{\longrightarrow} 0 \, ,
\end{displaymath}
which yields interpolating with \eqref{aprioriestimate.1} 
\begin{equation} \label{existence.1}
\|u^{\mu}-u^{\mu'}\|_{L^{\infty}_1H^{s_-}}\underset{\mu',\mu \to 0}{\longrightarrow} 0 \, ,
\end{equation}
for any number $s_-$ slightly smaller than $s$.

\smallskip

Therefore, there exists $u \in C([0,1] : H^{s_-}(M))$ such that $\{u^{\mu}\}$ converges to $u$ in $L^{\infty}([0,1] : H^{s_-}(M))$ as $\mu \to 0$. Passing to the limit in \eqref{dissiphomodel} as $\mu \to 0$ , it is easy to verify that $u$ is a solution to \eqref{homodel} in the distributional sense.

\subsection{Uniqueness} \label{Un} Let $u_1$ and $u_2$ be two solutions of \eqref{homodel} in $C([0,T] : H^s(M))$ corresponding to the same initial datum $u_1(\cdot,0)=u_2(\cdot,0)=u_0 \in H^s(M)$. As explained above, it is sufficient to assume that $\|u_0\|_{H^s} \le \delta$, where $\delta$ is a small positive number, and that $u_1$, $u_2$ are defined on a time interval $[0,1]$. 
\quad

Arguing as in the proof of Lemma \ref{aprioriestimate} (using Proposition \ref{EE} instead of Proposition \ref{diffapriorienergy}), we deduce the \textit{a priori} estimate 
\begin{displaymath} 
\|u_1\|_{L^{\infty}_1H^s_x}+\|u_2\|_{L^{\infty}_1H^s_x} \lesssim \delta \, .
\end{displaymath}
Let $v=u_1-u_2$. We conclude that $v \equiv 0$ on $[0,1]$ by integrating \eqref{diffEE.4} and using \eqref{diffEE.3}, for  $\delta$ chosen small enough.
 
\subsection{Persistence property and continuity of the flow map} \label{PE}
In this subsection, we will use the Bona-Smith argument \cite{BoSm} in order to prove the persistence property, that is $u \in C([0,T] : H^s(\mathbb R))$, and the continuity of the flow map. 

\smallskip
Let $u_0 \in H^s(M)$. By using a scaling argument, we can always assume that $\|u_0\|_{H^s} < \delta$, where $\delta$ is a small positive number. From the existence part, the corresponding solution $u$ to \eqref{homodel} is defined on $[0,1]$ and belongs to $C([0,1] : H^{s_-}(M))$, where $s_-$ denotes any number slightly smaller than $s$.

\smallskip
Now, we regularize the initial datum $u_0$ and consider the  corresponding IVP 
\begin{equation} \label{reghomodel}
\left\{
\begin{array}{l}\partial_tu^{\epsilon}+\partial_x^{2l+1}u^{\epsilon}=u^{\epsilon}\partial_x^{2l-1}u^{\epsilon} \\ u^{\epsilon}(\cdot,0)=u_0^{\epsilon}=u_0\ast \rho_{\epsilon} \in H^{\infty}(\mathbb R) \, ,
\end{array} \right. 
\end{equation}
where $\rho_{\epsilon}$ is an approximation of the identity. 

More precisely, let $\rho \in \mathcal{S}(\mathbb R)$ if $M=\mathbb R$, respectively $\rho \in C^{\infty}_{per}$ if $M=\mathbb T$, be such that $\int\rho(x) \, dx=1$ and $\int x^k\rho(x) \, dx=0,$ for $k \in \mathbb Z_{+}$ with $0 \le k \le [s]+1$. For any $\epsilon>0$, define $\rho_{\epsilon}(x)=\epsilon^{-1}\rho(\epsilon^{-1}x)$. The following lemma, whose proof can be found in \cite{BoSm} (see also Proposition 2.1 in \cite{KaPo2}), gathers the properties of the smoothing operators
\begin{lemma} \label{BonaSmith}
Let $s \ge 0$, $\phi \in H^s(M)$ and for any $\epsilon>0$,
$\phi_{\epsilon}=\rho_{\epsilon} \ast \phi$. Then,
\begin{equation} \label{BonaSmith.1}
\|\phi_{\epsilon}\|_{H^{s+\nu}} \lesssim \epsilon^{-\nu}
\|\phi\|_{H^s}, \quad \forall \nu \ge0,
\end{equation}
and
\begin{equation} \label{BonaSmith.2}
\|\phi-\phi_{\epsilon}\|_{H^{s-\beta}} \underset{\epsilon
\rightarrow 0}{=} o(\epsilon^{\beta}), \quad \forall \beta \in
[0,s].
\end{equation}
\end{lemma}

\smallskip 
From the existence part, there exists a solution $u^{\epsilon}$ of \eqref{reghomodel} in $C([0,1] : H^{\infty}(M))$, for all $0<\epsilon \le 1$. Moreover, $\|u_0^{\epsilon}\|_{H^s} \le \|u_0\|_{H^s} \le \delta$. Thus, we deduce arguing as in the proof of Lemma \ref{aprioriestimate}, using Proposition \ref{EE} instead of Proposition \ref{apriorienergy}, that 
\begin{equation} \label{persistence.1}
\|u^{\epsilon}\|_{L^{\infty}_1H^s_x} \lesssim \|u_0^{\epsilon}\|_{H^s} \lesssim \delta \, .
\end{equation}

Now, let $0<\epsilon'<\epsilon \le 1$. By applying estimates \eqref{diffEE.3}-\eqref{diffEE.5} of Proposition \ref{diffEE} for $u_1=u^{\epsilon}$ and $u_2=u^{\epsilon'}$ and using \eqref{BonaSmith.1}, \eqref{BonaSmith.2} and \eqref{persistence.1} we get that
\begin{equation} \label{persistence.2}
\|u^{\epsilon}-u^{\epsilon'}\|_{L^{\infty}_1L^2_x} \lesssim \|u_0^{\epsilon}-u_0^{\epsilon'}\|_{L^2} 
\underset{\epsilon \to 0}{=}o(\epsilon^s)
\end{equation}
and
\begin{equation} \label{persistence.3}
\begin{split}
\|u^{\epsilon}-u^{\epsilon'}\|_{L^{\infty}_1H^s_x}^2 &\lesssim \|u_0^{\epsilon}-u_0^{\epsilon'}\|_{H^s}^2
\\&+ \left(\sum_{k=1}^{l-1}\|u^{\epsilon}(t)\|_{L^{\infty}_1H^{s+2l-k}_x}\right)\|u^{\epsilon}-u^{\epsilon'}\|_{L^{\infty}_{1,x}}\|u^{\epsilon}-u^{\epsilon'}\|_{L^{\infty}_1H^s_x}
\end{split}
\end{equation}
as soon as $\delta$ is chosen small enough. Then, we conclude combining the Gagliardo-Nirenberg inequality
\begin{displaymath} 
\|v\|_{L^{\infty}} \lesssim  \|v\|_{L^2}^{1-\frac1{2s}}\|v\|_{H^s}^{\frac1{2s}} \, ,
\end{displaymath}
with \eqref{persistence.1}-\eqref{persistence.3} that 
\begin{equation} \label{persistence.4}
\|u^{\epsilon}-u^{\epsilon'}\|_{L^{\infty}_1H^s_x}^2 \lesssim \|u_0^{\epsilon}-u_0^{\epsilon'}\|_{H^s}^2+\left(\sum_{k=1}^{l-1}\epsilon^{-(2l-k)}\right)o(\epsilon^{s-\frac12}) \underset{\epsilon \to 0}{\longrightarrow} 0 \, ,
\end{equation}
since $s>4l-\frac92$ and $l \ge 2$. 

Therefore, $\{u^{\epsilon}\}$ is a Cauchy sequence in $L^{\infty}([0,1] : H^s(M))$, so that there exists $\widetilde{u} \in C([0,1] : H^1(M))$ of \eqref{homodel} such that 
\begin{equation} \label{persistence.5}
u^{\epsilon} \underset{\epsilon \to 0}{\longrightarrow} \widetilde{u} \quad \text{in} \quad L^{\infty}_1H^s_x \, .
\end{equation} 
By using the uniqueness result, we conclude that $u \equiv \widetilde{u} \in C([0,1] : H^s(M))$.

\medskip 

Now, we turn to the proof of the continuous dependence. Let $\theta>0$ be given. It suffices to prove that there exists $\kappa>0$ such that for any $v_0 \in H^s(M)$ satisfying  $\|u_0-v_0\|_{H^s} < \kappa$, the solution $v \in C([0,1] : H^s(M))$  of \eqref{homodel} emanating from $v_0$ satisfies 
\begin{equation} \label{contdep.1}
\|u-v\|_{L^{\infty}_1H^s_x} < \theta \, .
\end{equation}

For any $\epsilon>0$, we regularize the initial datum $v_0$ by defining $v_0^{\epsilon}=v_0 \ast \rho_{\epsilon}$ as above. Then, it follows from the triangle inequality that 
\begin{equation} \label{contdep.2}
\|u-v\|_{L^{\infty}_1H^s_x} \le \|u-u^{\epsilon}\|_{L^{\infty}_1H^s_x}+
\|u^{\epsilon}-v^{\epsilon}\|_{L^{\infty}_1H^s_x}
+\|v-v^{\epsilon}\|_{L^{\infty}_1H^s_x} \, .
\end{equation}
According to \eqref{persistence.5}, we can find $\epsilon_0>0$ small enough such that 
\begin{equation} \label{contdep.3}
\|u-u^{\epsilon_0}\|_{L^{\infty}_1H^s_x}+\|v-v^{\epsilon_0}\|_{L^{\infty}_1H^s_x}<\theta/3 \, .
\end{equation}

In order to estimate the second term on the right-hand side of \eqref{contdep.2}, we consider the parabolic regularizations $u^{\epsilon_0,\mu}$ and $v^{\epsilon_0,\mu}$ of $u^{\epsilon_0}$ and $v^{\epsilon_0}$ for $0<\mu\le 1$, \textit{i.e.} $u^{\epsilon_0,\mu}$, respectively $v^{\epsilon_0,\mu}$, is a solution to the equation in \eqref{dissiphomodel} with initial datum $u_0^{\epsilon_0}$, respectively $v_0^{\epsilon_0}$. 

According to \eqref{BonaSmith.1}, $u_0^{\epsilon_0}$ and $v_0^{\epsilon_0}$ belong to $H^{s+1}(M)$ and satisfy 
\begin{displaymath} 
\|u_0^{\epsilon_0}\|_{H^{s+1}} \lesssim  \epsilon_0^{-1}\|u_0\|_{H^s} \lesssim \epsilon_0^{-1}\delta
\end{displaymath}
and 
\begin{displaymath} 
\|v_0^{\epsilon_0}\|_{H^{s+1}} \le  \|u_0^{\epsilon_0}\|_{H^{s+1}}
+\|u^{\epsilon_0}_0-v_0^{\epsilon_0}\|_{H^{s+1}}
\lesssim \epsilon_0^{-1}\delta+\epsilon_0^{-1}\kappa \, .
\end{displaymath}
Thus, by choosing $\delta=\delta(\epsilon_0)$ and $\kappa=\kappa(\epsilon_0)$ small enough, we deduce from the theory in Subsections \ref{APE}, \ref{EX} and \ref{Un} with $\sigma=s+1$ and $\sigma_-=s$, that $\{u^{\epsilon_0,\mu}\}_{\mu}$ and $\{v^{\epsilon_0,\mu}\}_{\mu}$ converge to $u^{\epsilon_0}$ and $v^{\epsilon_0}$  in $L^{\infty}([0,1] : H^s(M))$, as $\mu$ tends to $0$. Then, there exists $\mu_0>0$, small enough such that 
\begin{align} \label{contdep.4}
\|u^{\epsilon_0}-v^{\epsilon_0}\|_{L^{\infty}_1H^s_x} &\le \|u^{\epsilon_0}-u^{\epsilon_0,\mu_0}\|_{L^{\infty}_1H^s_x}+
\|u^{\epsilon_0,\mu_0}-v^{\epsilon_0,\mu_0}\|_{L^{\infty}_1H^s_x}
+\|v^{\epsilon_0}-v^{\epsilon_0,\mu_0}\|_{L^{\infty}_1H^s_x} \nonumber\\ &
\le \theta/3+\|u^{\epsilon_0,\mu_0}-v^{\epsilon_0,\mu_0}\|_{L^{\infty}_1H^s_x} \, .
\end{align}

Finally, observe from \eqref{BonaSmith.1} that $\|u_0^{\epsilon_0}-v_0^{\epsilon_0}\|_{H^s} < \kappa$. Then, we deduce from the continuous dependence of the regularized problem \eqref{dissiphomodel} in Proposition \ref{parabolic}, that 
\begin{equation} \label{contdep.5}
\|u^{\epsilon_0,\mu_0}-v^{\epsilon_0,\mu_0}\|_{L^{\infty}_1H^s_x} <\theta/3 \, ,
\end{equation}
if $\kappa=\kappa(\epsilon_0,\mu_0)$ is chosen small enough. 

Therefore \eqref{contdep.1} follows gathering \eqref{contdep.2}-\eqref{contdep.4}, which concludes the proof of the continuous dependence.

\vspace{0,5cm}

\noindent \textbf{Acknowledgments.} 
Part of this research was carried out when the second author was visiting the Department of Mathematics of the University of Chicago, whose hospitality is gratefully acknowledged. The authors would like to thank Gustavo Ponce for fruitful discussions about this work.

\end{document}